\documentclass{article}
\usepackage{a4wide}

\usepackage{amssymb,amsmath}
\usepackage{amsthm}
\usepackage{graphicx}
\usepackage[english]{babel}
\usepackage{relsize}
\usepackage{color}
\usepackage{enumitem}
\usepackage{color}

\def\rr{\mathbb{R}} 
\def\sF{\mathcal{F}} 
\def\Pr{\mathbb{P}} 
\def\EX{\mathbb{E}}	
\def\dif{\mathrm{d}} 
\def\landau{\mathcal{O}} 
\def\part{\mathcal{T}_{\mathbf{h}}^N} 
\def\e{\textnormal{e}} 
\def\eps{\epsilon}

\def\nhat{\widehat}
\def\nbar{\overline}

\def\tn{\textnormal}

\theoremstyle{plain}
\newtheorem{theorem}{Theorem}[section]

\newtheorem{lemma}[theorem]{Lemma}

\theoremstyle{definition}
\newtheorem{definition}[theorem]{Definition}
\newtheorem{remark}[theorem]{Remark}

\numberwithin{equation}{section}

\begin{document}

\title{Strong Error Analysis of the $\Theta$-Method for Stochastic Hybrid Systems}

\author{Martin G.~Riedler\footnotemark[1]$~$ and Girolama Notarangelo\footnotemark[2]}

\renewcommand{\thefootnote}{\fnsymbol{footnote}}
\footnotetext[1]{%
martin.riedler@jku.at\\[2pt]
Institute for Stochastic,
Johannes Kepler University Linz,
Altenbergerstra\ss e 69,
4040 Linz, Austria.} 
\footnotetext[2]{%
girolama.notarangelo@jku.at\\[2pt]
Institute for Stochastic,
Johannes Kepler University Linz,
Altenbergerstra\ss e 69,
4040 Linz, Austria.}
\renewcommand{\thefootnote}{\arabic{footnote}}

\maketitle

\begin{abstract}
We discuss numerical approximation methods for Random Time Change equations which possess a deterministic drift part and jump with state-dependent rates. It is first established that solutions to such equations are versions of certain Piecewise Deterministic Markov Processes. Then we present a convergence theorem establishing strong convergence (convergence in the mean) for semi-implicit Maruyama-type one step methods based on a local error analysis. The family of $\Theta$--Maruyama methods is analysed in detail where the local error is analysed in terms of It{\^o}-Taylor expansions of the exact solution and the approximation process. The study is concluded with numerical experiments that illustrate the theoretical findings.
\end{abstract}

%

\section{Introduction}

In this study we consider the convergence of numerical approximations for Random Time Change Equations (RTEs) of the form
\begin{equation}\label{random_time_change_eq}
X(t)=X(0)+\int_0^t f(X(s))\,\dif s + \sum_{k=1}^pY_k\biggl(\int_0^t \lambda_k(X(s))\,\dif s\biggr)\,\nu_k\,,
\end{equation}
where $(Y_k(t))_{t\geq 0}$, $k=1,\ldots,p\in\mathbb{N}$, are independent unit rate Poisson processes on a complete probability space $(\Omega,\sF,\Pr)$ and $\nu_k\in\rr^d$ are the jump heights. The coefficient functions $f:\rr^d\to\rr^d$ and $\lambda_k:\rr^d\to\rr_+$ are Lipschitz continuous, bounded and in general sufficiently smooth such that all occurring derivatives exist and are bounded.
Here the first coefficient $f$ governs the continuous movement of the system and the remaining coefficients $\lambda_k$ are state-dependent instantaneous jump rates. A unique pathwise solution $(X(t))_{t\geq 0}$ of \eqref{random_time_change_eq} exists under these assumptions on the coefficients. The probability space $(\Omega,\sF,\Pr)$ is equipped with a suitable filtration $(\mathcal{G}_t)_{t\geq 0}$ to which the solution $(X(t))_{t\geq 0}$ of \eqref{random_time_change_eq} is adapted and such that it satisfies a strong Markov property. This filtration is specially constructed and we refer to Section \ref{section_set_up} for the details. The main results of the present study are the following.

First, we show that RTEs \eqref{random_time_change_eq} are equivalent in law to certain Piecewise Deterministic Markov Processes (PDMPs) by studying the martingale problem for these processes. This class of PDMPs are of particular importance for modelling chemical reaction networks with multiple, clearly separated time scales \cite{Debussche}. Examples of PDMPs in (bio-)chemical applications in recent years are, e.g., models of excitable biological membranes \cite{BuckwarRiedler,ClayDefelice,Wainrib1,RudigerFalcke}, biochemical reaction systems exhibiting multiple time-scales \cite{Huisinga,Haseltine,Kalantzis} or gene regulatory networks \cite{Zeisleretal}. Additional to these applications in bioscience PDMP models are also of high interest in fields such as control and queueing theory, mathematical finance and ecology, cf.~\cite{Davis1,Davis2,Jacobsen,Hanson} and references therein. The advantage of these models being cast in the form of RTEs \eqref{random_time_change_eq} over their PDMP formulation is that the former is substantially better suited for theoretical analysis. The method of representing stochastic processes as random time changes to certain standard types of processes, e.g., Poisson processes in our case, see \cite{EthierKurtz} for the general concept, was recently employed in a series of studies to analyse tau-leaping methods for chemical reactions systems in \cite{AndGanKurtz,AndersonHigham,AndersonKoyama}.

Secondly, we present convergence theorems for the strong convergence of non-jump adapted, semi-implicit Maruyama-type methods. That is, for approximations $\nhat X_n$ to $X(t_n)$ on a grid $0=t_0<t_1<\ldots t_{\nbar n}=T$ we are concerned with the global error in the mean
\begin{equation}\label{def_strong_global_error}
\max_{n=1,\ldots,\nbar n}\EX|X(t_n)-\nhat X_n|\,.
\end{equation}
The convergence of the approximation methods is obtained via an analysis of the local error. Employing the general convergence theorems we then conduct a strong error analysis of the $\Theta$-Maruyama method for equation \eqref{random_time_change_eq}. 
We note that the importance of the strong error analysis, besides being of mathematical interest in itself, lies particularly in the fact of its necessity for Multi-level Monte Carlo methods. We believe that the methods in the present paper can be used as a guideline to derive Multi-level Monte Carlo methods for statistics of interest similar as in the case of pure-jump processes \cite{AndersonHigham}. 
In general numerical methods for jump processes, e.g., pure jump processes, PDMPs or jump-diffusions, can be grouped into two classes, one which aims to resolve the jump times and heights exactly and a second which only approximates in a given time interval the number of jumps and their cumulative height. The former class, called jump-adapted methods, contains the Stochastic Simulation Algorithm (SSA), viz.~Gillespie's method, for chemical reaction systems, viz.~continuous time Markov chains, jump adapted methods for JSODEs, see, e.g., the recent monograph \cite{PlatenLiberati}, and a recently presented jump-adapted method for PDMPs \cite{RiedlerAS_PDMP}. The second class contains fixed time step approximation methods for chemical reaction systems, called tau-leaping methods in this community, and jump-diffusions. To the best of our knowledge fixed time-step methods for PDMPs have not been considered yet. Thus due to the equivalence of RTEs \eqref{random_time_change_eq} to PDMPs the present paper fills the gap and considers fixed time-step methods for PDMPs.
We note that the aptness, applicability and efficiency of the two classes of methods depends on the equation under consideration and in particular on the relationship between the time-step and the jump intensity. Roughly speaking, slow, large jumps should be resolved by jump adapted methods and fast, small jumps by fixed time step methods. However, these are essentially issues regarding particular models and thus they are not within the general scope of the present study.\medskip

The remainder of the study is organised as follows. In Section \ref{section_set_up} we present a technically detailed description of solutions to equations \eqref{random_time_change_eq}, prove their existence and uniqueness and establish their equivalence to PDMPs. Furthermore, in this section we derive technical tools for the convergence analysis to follow. The detailed and technical proofs of this section are mainly deferred to the appendix. We introduce in Section \ref{section_conv_theorems} semi-implicit approximation methods for equation \eqref{random_time_change_eq} and the main convergence result is established. Next, in Section \ref{section_theta_method} we conduct a detailed strong error analysis of $\Theta$-Maruyama methods. The theoretical findings of the preceding sections are illustrated on numerical examples in Section \ref{section_numerical_examples}. The study is closed by drawing conclusions in Section \ref{section_conclusions} wherein we also briefly comment on some further related work.

\section{PDMPs represented as a Random Time Change Equations}\label{section_set_up}

In the introduction we have already stated the first main result that solutions of RTEs of the form \eqref{random_time_change_eq} are versions of certain PDMPs. This family of c\`adl\`ag strong Markov processes, introduced in \cite{Davis1,Davis2}, a more recent extended discussion can be found in \cite{Jacobsen}, combines deterministic continuous movement with random jump events. These two dynamics are coupled in a way such that the deterministic dynamics in between jumps depend on the outcomes of the preceding jumps, and the jump intensities and jump heights depend on the current continuously changing paths. In a simplifying but for the present study relevant case, one can think of such processes as generalisations of continuous time Markov chain, which are not constant in between jumps but follow a path given by the solution of an ODE. This structure is made clear by \emph{Davis' Construction Procedure} which is used in \cite{Davis1,Davis2} to define  such processes but remains valid for more general classes of PDMPs \cite{RiedlerPhD}. In the present study the general class of PDMPs is restricted so that we obtain an equivalence in law to the solution of \eqref{random_time_change_eq}. Thus, we assume that the jump dynamics are such that there occur almost surely only finitely many different jump heights which are given by $\nu_1, \ldots \nu_p \in\rr^d$. Further we set the total instantaneous jump rate $\lambda(x)=\sum_{k=1}^p\lambda_k(x)$ and define a Markov kernel $\mu$ on $\rr^d$ by the probabilities
\begin{equation*}
\mu(x,\{y\})=\left\{\begin{array}{cl}\displaystyle\frac{\lambda_k(x)}{\lambda(x)}& \textnormal{if } y=x+\nu_k \textnormal{ for some } k=1,\ldots,p,\\[2ex] 0 &\textnormal{otherwise}.\end{array}\right.
\end{equation*}
Then, classically, a PDMP $(X_\tn{PDMP}(t))_{t\geq 0}$ is defined by the following constructive method which we present in pseudo-code form.

\begin{center}
\begin{minipage}{0.9\textwidth}
\textbf{Start:} Set $\tau_0=0$ and sample a realisation of the initial condition $X_\tn{PDMP}(0)$.\\[1ex]
\textbf{For $n\geq 1$:} The $n$th jump time $\tau_n$ is a realisation of a random variable with distribution
\begin{equation*}
\Pr\bigl[\tau_n>t+\tau_{n+1}\bigr]\,=\,\exp\Bigl(-\int_0^t \phi(s,X_\tn{PDMP}(\tau_{n-1}))\,\dif s\bigr)\qquad\forall\,t\geq 0
\end{equation*}
and set
\begin{equation*}
 X_\tn{PDMP}(s)=\phi(s-\tau_{n-1},X_\tn{PDMP}(\tau_{n-1})) \textnormal{ for } s\in (\tau_{n-1},\tau_n)\,.
\end{equation*}
Sample a realisation of the post-jump value
\begin{equation*}
X_\tn{PDMP}(\tau_n)\sim \mu\bigl(\phi(\tau_n-\tau_{n-1},X_\tn{PDMP}(\tau_{n-1})),\cdot\bigr)\,. 
\end{equation*}
\textbf{End}
\end{minipage}
\end{center}

For a more inclusive and thorough discussion of PDMPs we refer to the monographs \cite{Davis2,Jacobsen} and the thesis of one of the present authors \cite{RiedlerPhD}. Here we only emphasise the fundamental difference of PDMPs to jump-diffusions without a Brownian motion part. For PDMPs the jump intensities are in general state-dependent and thus continuously changing in time whereas for jump-diffusions they are constant.

We close this discussion with the most important result (at present) for the characterisation of PDMPs. This result allows the identification of PDMPs with solutions of equation \eqref{random_time_change_eq}. For any function $F\in C^1_b(\rr^d)$, i.e., the set of bounded continuously differentiable functions from $\rr^d$ to $\rr$ with bounded first derivatives, the generator of a PDMP of the above type is given by\footnote{To clarify some notation used throughout the study we note that we use the notation $\nabla F[f]$ to indicate the application of the derivative of $F$ to $f$, where we omit or abbreviate the arguments of the functions whenever possible for the sake of simplicity. That is, in more detail, we use
\begin{equation*}
\nabla F[f](x)\,=\, \nabla F(x)[f(x)]\,=\,\sum_{j=1}^d \frac{\partial F}{\partial x_j}(x)\,f_j(x)\,.
\end{equation*}
}
\begin{eqnarray}\label{generator_pdmp}
\mathcal{A}F(x)&=&\nabla F[f](x)+\lambda(x)\int_{\rr^d} \Bigl(F(\xi)-F(x) \Bigr)\,\mu(x;\dif\xi)\nonumber\\[2ex] &=&\nabla F[f](x)+\sum_{k=1}^p\lambda_k(x)\,\Bigl(F(x+\nu_k)-F(x) \Bigr),
\end{eqnarray}
see \cite{Jacobsen}. It was proven in \cite[Thm.~2.5]{Debussche} (see also the appendix, where we present the result for the relevant setting in this study) that the martingale problem posed by this generator possesses a unique solution. The fact that a solution of equation \eqref{random_time_change_eq} is a version of the PDMP is established in the following theorem, stating that it solves the martingale problem posed by the generator of the PDMP. The proof of the theorem is deferred to the appendix.

\begin{theorem}\label{theorem_equivalence}  Equation \eqref{random_time_change_eq} possesses for almost all $\omega\in\Omega$ a unique c\`adl\`ag solution. Furthermore, this solution solves the martingale problem posed by the generator \eqref{generator_pdmp} and any $F\in C^1_b(\rr^d)$. 
\end{theorem}

We continue discussing immediate consequences of the above results. In particular we define a suitable filtration $(\mathcal{G}_t)_{t\geq 0}$ to which the solution $X$ is adapted. We define for arbitrary multi-indices $u=(u_1,\ldots,u_p)\in\rr_+^p$ the $\sigma$-fields
\begin{equation*}
\sF_u=\sigma(Y_k(s),s\leq u_k; k=1,\ldots,p)\vee\sigma(\mathcal{N}),
\end{equation*}
where $\mathcal{N}$ is the collection of all $\Pr$-null sets in $\sF$. Note, that the filtration is complete. For background on filtrations with respect to partially ordered sets and multi-parameter stopping times, we refer to \cite[Chap.~2.8]{EthierKurtz} which contains all subsequently necessary results. Clearly, the processes $Y_k$ are Poisson processes with respect to the multi-valued filtration $(\sF_u)_{u\in\rr_+^p}$. Then it holds due to the uniqueness of the solution to \eqref{random_time_change_eq} and the completeness of the filtration that for all $k$ and $t\geq 0$
\begin{equation}\label{definition_local_time}
 \tau_k(t):=\int_0^t\lambda_k(X(s))\,\dif s 
\end{equation}
is a stopping time with respect to the filtration $(\sF_u)_{u\in\rr_+^p}$, cf.~\cite[Thm.~2.2(a), Chap.~6]{EthierKurtz}. Furthermore, also the random times
\begin{equation*}
\tau_{j+}(t):=\int_0^t f_j(X(s))_{+}\,\dif s,\qquad \tau_{j-}(t):=\int_0^t f_j(X(s))_{-}\,\dif s
\end{equation*} 
are stopping times with respect to $(\sF_u)_{u\in\rr_+^p}$, where the subscripts $+$ and $-$ refer to the positive and negative parts of $f_j(X(s))$. Note that these stopping times are bounded due to the boundedness assumption on the coefficient functions. Next we define the multi-parameter stopping time $\tau(t)=\bigl(\tau_1(t),\ldots,\tau_p(t),\tau_{1+}(t),\ldots,\tau_{d+}(t),\ldots,\tau_{1-}(t),\tau_{d+}(t)\bigr)$ which then in turn defines a complete filtration 
\begin{equation*}
\mathcal{G}_t:=\sF_{\tau(t)}\,.
\end{equation*}
As stopping times and  c\`adl\`ag processes evaluated at stopping times are measurable with respect to the stopped $\sigma$-field it holds that $X$ is adapted to $(\mathcal{G}_t)_{t\geq 0}$, hence $\sF^X_t\subseteq\mathcal{G}_t$.
Finally, due to the Optional Stopping Theorem it follows that for each $k$ the process
\begin{equation}\label{app_poisson_martingales}
t\mapsto Y_k\Bigl(\int_0^t\lambda_k(X(s))\,\dif s\Bigr)-\int_0^t\lambda_k(X(s))\,\dif s
\end{equation}
is a $\mathcal{G}_t$-martingale.

\subsection{An It{\^o}-formula for Random Time Change Equations}

An important tool for the analysis of stochastic equations in order to obtain a type of Taylor expansion is an It{\^o}-formula. The following lemma contains an It{\^o}-formula for RTEs \eqref{random_time_change_eq}.

\begin{lemma}\label{lem_ito_formula} For every $F\in C^1_b(\rr^d)$ it holds almost surely that
\begin{eqnarray}\label{ito_formula}
\lefteqn{F(X(t))-F(X(0))\ =}\\[2ex]
&=&\int_0^t \nabla F[f](X(s))\,\dif s +\int_0^t\sum_{k=1}^p\lambda_k(X(s))\Bigl(F(X(s)+\nu_k)-F(X(s))\Bigr)\,\dif s + M^F(t)\,.\nonumber
\end{eqnarray}
The process $(M^F(t))_{t\geq 0}$, which depends on $F$, is a c\`adl\`ag martingale 
with respect to the filtration $(\mathcal{G}_t)_{t\geq 0})$. It satisfies $M^F(0)=0$ a.s., $\EX M^F(t)=0$ for all $t\geq 0$ and the second moments of its increments satisfy
\begin{equation}\label{def_quad_variation}
\EX|M^F(t)-M^F(s)|^2\,=\,\EX\int_s^{t}\sum_{k=1}^p\lambda(X(r))\,\big|F(X(r)+\nu_k)-F(X(r))\bigr|^2\,\dif r\qquad\forall\, 0\leq s\leq t\,.
\end{equation}
\end{lemma}

\begin{proof} First we take \eqref{ito_formula} as a definition of the c\`adl\`ag process $M^F$ which implies $M^F(0)=0$ almost surely. Next, note that the assumption $F\in C^1_b(\rr^d)$ implies that $F$ is in the domain of the generator of the process and thus $M^F$ is a martingale with respect to the natural filtration \cite[Thm.~7.6.1]{Jacobsen}. Hence, $\EX M^F(t)=0$ for all $t\geq 0$ follows and, in particular, the second moments of the increments \eqref{def_quad_variation} follow from \cite[Prop.~4.6.2]{Jacobsen}. 
\end{proof}

\begin{remark} In this study we use the It{\^o} formula \eqref{ito_formula} for vector-valued functions $F=(F_1,\ldots,F_d)$ $\in C^1_b(\rr^d,\rr^d)$  without specific distinction in notation. In this case \eqref{ito_formula} is simply employed component-wise such that $\mathcal{A}F=(\mathcal{A}F_1,\ldots,\mathcal{A}F_d)$ and $M^F=(M^{F_1},\ldots,M^{F_d})$.
\end{remark}


\section{Approximation methods for Random Time Change Equations}\label{section_conv_theorems}

In this section we present the general class of one-step approximation methods for RTEs \eqref{random_time_change_eq}. Our aim is to approximate the solution $(X(t))_{t\geq 0}$ on an equidistant grid with step-size $h$ over a fixed time interval $[0,T]$ by random variables $\nhat X_n$, $n=1,\ldots,\nbar n$, with $T=h\nbar n$, such that $X(t_n)\approx \nhat X_n$ for $t_n=nh$. Then for $h\to 0$ these methods should converge strongly to the exact solution, that is for $h\to 0$ the global error \eqref{def_strong_global_error} vanishes. In analogy to (J)SDEs we first define the generic class of semi-implicit one-step methods.

\begin{definition} A \emph{semi-implicit one-step method} is of the generic form $\nhat X_0= X(0)$ and the iteration
\begin{equation}\label{def_semi_impl_meth}
\nhat X_{n+1}\,=\, \nhat X_n+h\phi_1(\nhat X_n,\nhat X_{n+1},h)+\phi_2(\nhat X_n,h,\nhat\tau(t_n);Y),
\end{equation}
where $\nhat \tau(t)=(\nhat \tau_1(t),\ldots,\nhat \tau_p(t))$ denotes the approximation of local times \eqref{definition_local_time} for the driving Poisson processes $Y$ which are given by $\nhat \tau_k(0)=0$ and the iterations
\begin{equation}\label{def_semi_impl_meth_2}
\nhat\tau_k(t_{n+1})=\nhat\tau_k(t_n)+h\phi_3(k,\nhat X_n,h)\,.
\end{equation}
Here $\phi_1$ and $\phi_2$ are used to approximate the drift and the jump part of the RTE and $\phi_3$ approximates the integral in the time change. A method is called \emph{explicit} if $\phi_1$ does not depend on $\nhat X_{n+1}$. 
If $\phi_2$ in \eqref{def_semi_impl_meth} is of the form
\begin{equation}\label{def_maruyama_term}
\phi_2(\nhat X_n,h,\nhat\tau(t_n);Y)=\sum_{k=1}^p \biggl(Y_k\bigl(\nhat\tau_k(t_n)+h\phi_3(k,\nhat X_n,h)\bigr)-Y_k\bigl(\nhat\tau_k(t_n)\bigr)\Biggl)\,\nu_k
\end{equation}
then the method is called a \emph{Maruyama-type method}. We note that for numerically solving the RTE \eqref{random_time_change_eq} in actual implementations Maruyama-type methods are straight-forward to implement. The strong Markov property of the Poisson processes implies that the Poissonian increments $Y_k(\nhat \tau_k(t_{n+1}))-Y_k(\nhat \tau_k(t_n))$ are conditionally independently Poisson distributed with rate $\phi_3(k,\nhat X_n,h)$ as the times $\nhat\tau(t_n)$ are stopping times. Hence simulating one path of \eqref{random_time_change_eq} necessitates only the simulation of $p\nbar n$ Poisson random variables.
\end{definition}

\begin{remark} In the definition of the method \eqref{def_semi_impl_meth} and \eqref{def_semi_impl_meth_2} the function $\phi_1$ is an increment function of a (possibly implicit) one-step method for ordinary differential equations, and the fact that $\nhat\tau_k(t_n)$ is an approximation of the integral \eqref{definition_local_time}, i.e.,
\begin{equation*}
\int_{t_n}^{t_{n}+h}\lambda_k(X(s))\,\dif s\approx h\phi_3(k,\nhat X_n,h)\,,
\end{equation*} 
suggests that $\phi_3$ is an increment function of a deterministic one-step quadrature scheme. Finally, in \eqref{def_semi_impl_meth} the term $\phi_2$ generally refers to an approximation of stochastic integral terms as in (J)SDEs, e.g., a Maruyama-type term \eqref{def_maruyama_term} or in a future extension a Milstein-type term to obtain higher strong order methods.
\end{remark}

In the remainder of the study we focus on Maruyama-type methods, thus we always assume that $\phi_2$ is of the form \eqref{def_maruyama_term}. For concrete examples of such methods we refer to Section \ref{section_theta_method}, where $\Theta$-methods are considered and proceed with the discussion that generally allow for an error analysis of such methods. Along the lines of standard procedures in numerical analysis of differential equations we proceed deriving convergence from an analysis of the local error of a method, i.e., the error incurred over a single approximation step. To this end we first define the \emph{strong local errors} 
\begin{equation}\label{local_error_drift}
L_{n+1}:=\int_{t_n}^{t_{n+1}} f(X(s))\,\dif s - h\phi_1(X(t_n),X(t_{n+1}),h)
\end{equation}
and 
\begin{equation}\label{local_error_local_time}
K_{n+1}:=\sum_k \Bigl(\int_{t_n}^{t_{n+1}}\lambda_k(X(s))\,\dif s -h\phi_3(k,X(t_n),h)\Bigr)\,\nu_k\,.
\end{equation}
We say a numerical method is \emph{strongly consistent} if for some $q>0$ the strong local errors satisfy
\begin{equation*}
\max_{n=1,\ldots,\nbar n} \EX|L_n| = \landau(h^{1+q}),\qquad \max_{n=1,\ldots,\nbar n} \EX|K_n| = \landau(h^{1+q})\,.
\end{equation*}

\begin{remark} The errors $L_{n+1},\, K_{n+1}$ are the errors arising in one approximation step starting from the exact solution when approximating the drift and the local time of the Poisson processes, respectively. The current concept of strong consistency is closely related to the usual consistency condition in stochastic and deterministic numerical analysis. There it is normally assumed that the `classical' local error, i.e., the one-step approximation
\begin{equation*}
L(X(t_n),h):=X(t_n+h)-X(t_n)-h\phi_1(X(t_n),X(t_n+h),h)+\phi_2(X(t_n),h,\tau(t_n);Y),
\end{equation*}
satisfies $\EX |L(X(t_n),h)|=\landau(h^{q+1})$. However, it is easily observed that in the present case it holds for Maruyama-type methods that $\EX |L(X(t_n),h)|\leq \EX |L_{n+1}|+\EX|K_{n+1}|$.
\end{remark}

The subsequent central theorem contains the step from the local to the global error. Thus it allows to conclude from consistency of the numerical method to convergence.

\begin{theorem}[\textnormal{Convergence Theorem for Maruyama-type Methods}]\label{convergence_theorem} Under the Lipschitz conditions
\begin{eqnarray}
\big|\phi_1(x_1,y_1,t,h)-\phi_1(x_2,y_2,t,h)\big|&\leq& L_1|x_1-x_2|+L_2|y_1-y_2|\label{lipschitz_cond_1}\\[2ex]
\big|\phi_3(k,x_1,h)-\phi_3(k,x_2,h)\big|&\leq& L_{3,k}\,|x_1-x_2|\quad\forall\,k=1,\ldots,p,\label{lipschitz_cond_2}
\end{eqnarray}
and, if the method is implicit, the step-size restriction 
\begin{equation}\label{step_size_restriction}
h\,<\,\frac{1}{L_2} 
\end{equation}
it holds that a Maruyama-type method \eqref{def_semi_impl_meth} -- \eqref{def_maruyama_term} possesses a unique solution a.s.~and is convergent in the strong sense, if it is strongly consistent.\medskip

In particular, if $X(0)=\nhat X(0)$, then the \emph{strong global error} satisfies
\begin{equation}\label{strong_stability_estimate}
\max_{n=1,\ldots,\nbar n}\EX|X(t_n)-\nhat X_n|\ \leq\ S\,h^{-1}\max_{n=1,\ldots,\nbar n}\bigl(\EX|L_{n}|+\EX|K_n|\bigr)
\end{equation}
for a suitable constant $S>0$ independent of the initial condition. 
\end{theorem}

\begin{remark}\label{remark_stability_theorem} 
For a more detailed convergence analysis we are also interested in the asymptotic of the pre-factor $S$ in the stability estimates. Thus we note that the constant $S$ in \eqref{strong_stability_estimate} depends exponentially on the Lipschitz constants of the coefficients, i.e., $S=\landau\bigl(\exp\bigl(L_1+L_2+\sum_k L_{3,k}|\nu_k|\bigr)\bigr)$. 
\end{remark}

\begin{proof}[Proof of Theorem \ref{convergence_theorem}] If the method is semi-implicit we assume the step-size restriction \eqref{step_size_restriction} is satisfied, otherwise simply set $L_2=0$ in the subsequent proof and no step-size restriction is needed. In the following we first show (a) that the approximation method possesses a.s.~a unique solution and (b) that the random times $\nhat\tau(t_n)$ are stopping times. Then in part (c) we derive the stability estimate \eqref{strong_stability_estimate}. \medskip

\quad (a)\ We first consider the existence of a unique solution to the iteration scheme \eqref{def_semi_impl_meth}. If the numerical method is explicit, then this is trivially the case. In the case of a semi-implicit method the existence is proven under the step-size restriction \eqref{step_size_restriction} using a fixed point argument. For almost all $\omega\in\Omega$ the Poisson-increments are finite and given $\nhat X_n$ the next iterate $\nhat X_{n+1}$ is defined via the fixed-point equation
\begin{eqnarray*}
\nhat X_{n+1} &=& \nhat X_n+h\phi_1(\nhat X_n,\nhat X_{n+1},h)+\sum_{k=1}^p\Bigl(Y_k\bigl(\nhat\tau_k(t_n)+h\phi_3(k,\nhat X_n,h)\bigr)-Y_k\bigl(\nhat\tau_k(t_n)\bigr)\Bigr)\nu_k\\
&:=&\Gamma(\nhat X_{n+1};\nhat X_n,h,\Delta Y_n)\,,
\end{eqnarray*}
where $\Delta Y_n$ denote the Poissonian increments. In order to establish the existence of a unique fixed-point it remains to show that $y\mapsto\Gamma(y;x,h,w)$ is a contractive map for all sufficiently small $h$ uniformly in $x$ and $w$. The Lipschitz condition \eqref{lipschitz_cond_1} yields $|\Gamma(y_1;x,h,w)-\Gamma(y_2;x,h,w)|\ \leq\ hL_2|y_1-y_2|$ and hence $\Gamma$ is contractive under the step-size restriction \eqref{step_size_restriction}. 
\medskip

\quad (b)\ In this part of the proof we show that the random times $\nhat\tau_k(t_n)$ are $(\sF_u)_{u\in\rr_+^p}$-stopping times. We proceed by induction. First, it holds trivially that $\nhat\tau(t_0):=0$ is a stopping time and $\nhat X(0)$ is $\sF_{\nhat\tau(t_0)}$-measurable. Assume for $n\geq 1$ that $\nhat\tau_k(t_{n-1})$ is a stopping time and that $\nhat X(\nhat\tau(t_{n-1}))$ is $\sF_{\nhat\tau(t_{n-1})}$-measurable. Then the components $\nhat\tau(t_n)$ are defined as $\nhat\tau_k(t_n)=\nhat\tau_k(t_{n-1})+\phi_3(k,\nhat X_{n-1},h)$. As $\phi_3(k,\cdot,h)$ is a measurable function it follows that $\nhat\tau_k(t_n)$ is a stopping time. Then, it clearly holds that $\nhat X_{n}$ is $\sF_{\nhat\tau(t_{n})}$-measurable due to part (a) of the proof.
\medskip

\quad (c)\ In this part of the proof we derive the stability estimate \eqref{strong_stability_estimate}. Then inserting the definition of strong consistency immediately shows convergence of the method. Recursively inserting the method \eqref{def_semi_impl_meth} into its right hand side we obtain
\begin{eqnarray}
\nhat X_{n+1}&=& X_n + h\phi_1(\nhat X_n,\nhat X_{n+1},h)+\sum_{k=1}^p \Bigl( Y\bigl(\nhat \tau_k(t_{n+1})\bigr)-Y_k\big(\nhat\tau_k(t_n)\bigr)\Bigr)\nu_k\nonumber\\
&=& \nhat X_{n-1} +h\sum_{i=n-1}^n \phi_1(\nhat X_i,\nhat X_{i+1},h)+ \sum_{k=1}^p \Bigl( Y\bigl(\nhat \tau_k(t_{n+1})\bigr)-Y_k\big(\nhat\tau_k(t_{n-1})\bigr)\Bigr)\nu_k\nonumber\\
&=& \nhat X_0 + h\sum_{i=0}^n \phi_1(\nhat X_i,\nhat X_{i+1},h)+ \sum_{k=1}^p \Bigl( Y\bigl(\nhat \tau_k(t_{n+1})\bigr)-Y_k\big(\nhat\tau_k(t_0)\bigr)\Bigr)\nu_k\label{approx_method_expansion}
\end{eqnarray}
Thus subtracting \eqref{approx_method_expansion} from the exact process \eqref{random_time_change_eq} using that $\tau_k(t_0)=\nhat\tau_k(t_0)=0$ as $t_0=0$ and that $\nhat X_0=X(0)$ yields
\begin{eqnarray}
X(t_{n+1})-\nhat X_{n+1}&=& 
 h\sum_{i=0}^n \Bigl(\int_{t_i}^{t_{i+1}}f(X(s))\,\dif s-\phi_1(\nhat X_i,\nhat X_{i+1},h)\Bigr)\nonumber\\
&&\mbox{} +\sum_{k=1}^p \Bigl(Y_k\big(\tau_k(t_{n+1})\bigr)- Y\bigl(\nhat \tau_k(t_{n+1})\bigr)\Bigr)\nu_k\,.
\end{eqnarray}
Adding and subtracting the one-step approximations, using the notation \eqref{local_error_drift} and estimating the modulus with the triangle inequality yields
\begin{eqnarray}\label{strong_conv_prpoof_ineq_before_expect}
|X(t_{n+1})-\nhat X_{n+1}|&\leq& 
h\sum_{i=0}^{n}\Big|\phi_1(X(t_i),X(t_{i+1}),h)-\phi_1(\nhat X_i,\nhat X_{i+1},h)\Big|\nonumber\\
&&\mbox{}  +\sum_{k=1}^p \Big|Y_k\bigl(\tau_k(t_{n+1})\bigr)- Y\bigl(\nhat \tau_k(t_{n+1})\bigr)\Big|\,|\nu_k|+\sum_{i=0}^n |L_{i+1}|\,.
\end{eqnarray}
In the next step of the proof we take expectations on both sides of inequality \eqref{strong_conv_prpoof_ineq_before_expect}. We now first consider the Poisson process terms in more detail. The times $\tau_k(t_n)$ and $\nhat\tau_k(t_n)$ are stopping times with respect to the filtration $\sF_u$, see Section \ref{section_set_up} and part (b) of the proof, respectively. So, also the minimum $\tau_k(t_n)\wedge\nhat\tau_k(t_n)$ and the maximum $\tau_k(t_n)\vee\nhat\tau_k(t_n)$ are stopping times. Therefore, using that Poisson processes are non-decreasing and employing the Optional Stopping Theorem we obtain
\begin{eqnarray*}
\EX\Big|Y_k\big(\tau_k(t_{n+1})\bigr)- Y\bigl(\nhat \tau_k(t_{n+1})\bigr)\Big|&=&\EX\Bigl(Y_k\bigl(\tau_k(t_n)\vee\nhat\tau_k(t_n)\bigr)- Y\bigl(\tau_k(t_n)\wedge\nhat\tau_k(t_n)\bigr)\Bigr)\\[1ex]
&=&\EX\Bigl(\tau_k(t_n)\vee\nhat\tau_k(t_n)-\tau_k(t_n)\wedge\nhat\tau_k(t_n)\Bigr)\\[1ex]
&=&\EX\big|\tau_k(t_n)-\nhat\tau_k(t_n)\big|\,.
\end{eqnarray*}
Thus, the application of expectation to both sides of \eqref{strong_conv_prpoof_ineq_before_expect} and the addition the one-step approximation for the local time using the notation \eqref{local_error_local_time} yields the estimate
\begin{eqnarray*}
\lefteqn{\phantom{x}\hspace{-1cm}
\EX|X(t_{n+1})-\nhat X_{n+1}| }\nonumber\\
&\leq& 
h\sum_{i=0}^{n}\EX\Big|\phi_1(X(t_i),X(t_{i+1}),h)-\phi_1(\nhat X_i,\nhat X_{i+1},h)\Big|\nonumber\\
&&\mbox{}  +\sum_{k=1}^p \EX\Big|\int_0^{t_{n+1}} \lambda_k(X_s)\,\dif s - h\sum_{i=0}^n \phi_3(k,\nhat X_i,h)\Big|\,|\nu_k|+\sum_{i=0}^n \EX|L_{i+1}|\nonumber\\
&\leq& 
h\sum_{i=0}^{n}\EX\Big|\phi_1(X(t_i),X(t_{i+1}),h)-\phi_1(\nhat X_i,\nhat X_{i+1},h)\Big|\nonumber\\
&&\mbox{}  +h\sum_{k=1}^p |\nu_k|\Bigl(\sum_{i=0}^n \EX\Big|\phi_3(k,X(t_i),h)-\phi_3(k,\nhat X_i,h)\Big|\Bigr)+\sum_{i=0}^n \bigl(\EX|L_{i+1}|+\EX|K_{i+1}|\bigr)\,.
\end{eqnarray*}
We set $M:=h^{-1}\max_{i=0,\ldots\nbar n-1}\bigl(\EX|L_{i+1}|+\EX|K_{i+1}|\bigr)$ and $E_i:=\EX|X(t_i)-\nhat X_i|$ with $E_0:=0$ and obtain for all $n=0,\ldots \nbar n-1$ using the Lipschitz conditions \eqref{lipschitz_cond_1} and \eqref{lipschitz_cond_2} that
\begin{eqnarray*}
E_{n+1}&\leq& M+h(L_1+L_3)\sum_{i=0}^n E_i + hL_2\sum_{i=0}^n E_{i+1},
\end{eqnarray*}
where $L_3:=\sum_{k=1}^p L_{3,k}|\nu_k|$. Rearraging the right hand side in this inequality yields 
\begin{eqnarray*}
E_{n+1}&\leq& \frac{M}{1-hL_2}+\frac{h(L_1+L_2+L_3)}{1-hL_2}\sum_{i=0}^n E_i\,.
\end{eqnarray*}
We apply a discrete version of Gronwall's Lemma\footnote{Let $a_n$, $n\geq 0$, be a non-negative sequence, $c\in\rr$. If a real sequence $x_n$ satisfies for all $n\geq 0$ the inequality
\begin{equation*}
x_{n+1}\leq c+\sum_{i=0}^n a_ix_i
\end{equation*}
then it holds that
\begin{equation*}
x_{n+1}\leq (c+a_0x_0)\prod_{i=1}^n(1+a_i)\,.
\end{equation*}}
which results in the estimate
\begin{equation*}
E_{n+1}\,\leq\,\Bigl(\frac{M+h(L_1+L_2+L_3)E_0}{1-hL_2}\Bigr)\,\Bigl(1+\frac{h(L_1+L_2+L_3)}{1-hL_2}\Bigr)^n\,.
\end{equation*}
Finally, as $E_0=0$ we obtain for all $h\leq h^\ast\leq 1/L_2$ the estimate
\begin{equation}
E_{n+1}\,\leq\,\frac{M}{1-hL_2}\,\exp\Bigl(\frac{T(L_1+L_2+L_3)}{1-hL_2}\Bigr)\,\leq\, C\,\e^{CTL}\,M\,,
\end{equation}
where $L=L_1+L_3+L_3$ and the constant $C=(1-h^\ast L_2)^{-1}>1$ can be chosen as $C=1$ in the case of an explicit method. The proof of the stability estimate \eqref{strong_stability_estimate} is completed and strong convergence follows immediately from the strong consistency of the method.

\end{proof}

\section{Error analysis of the $\Theta$-Maruyama-methods}\label{section_theta_method}

In this section we present a detailed analysis of the global error of $\Theta$-Maruyama methods applying the convergence theorems established in the previous section. In practical applications often a scaling is inherent in the mathematical model, which is particular important for a practical relevant error analysis in chemical reaction systems. Taking the scaling into account in the error analysis resembles closely the analysis of numerical methods for small noise stochastic differential equations. Thus after presenting the $\Theta$-Maruyama methods and before stating the precise results we discuss this inherent scaling in a separate subsection.

\medskip

\begin{definition} The \emph{$\Theta$-Maruyama method} for step-sizes $h>0$ and a parameter $\theta\in[0,1]$ is given by $\nhat X_0=X_0$ and
\begin{equation}\label{def_theta_method}
\nhat X_{n+1}=\nhat X_n+h\theta f(\nhat X_{n+1})+h(1-\theta)f(\nhat X_n) +\sum_{k=1}^p \Bigl(Y_k\bigl(\nhat \tau_k(t_{n+1})\bigr)-Y_k\bigl(\nhat \tau(t_n)\bigr)\Bigr)\,\nu_k
\end{equation}
with $\nhat \tau(0)=0$ and
\begin{equation} 
\nhat \tau_k(t_{n+1})\, =\, \nhat\tau_k(t_{n})+h\,\phi_3^{\tn{method}}(k,\nhat X_n,h)\,,
\end{equation}
where $\phi_3^{\tn{method}}$ is the increment function of a suitable quadrature rule. The method \eqref{def_theta_method} is explicit if $\theta=0$ and then called \emph{explicit Euler method}. It is called \emph{implicit Euler method} in case of $\theta=1$ and \emph{trapezoidal method} for $\theta=\tfrac{1}{2}$. Reasonable choices for $\phi_3^{\tn{method}}$ are, e.g., the \emph{Euler quadrature rule}
\begin{eqnarray}\label{quad_euler}
\phi_3^{\tn{euler}}(k,x,h)& =& \lambda_k(x),
\end{eqnarray}
the \emph{midpoint quadrature rule}
\begin{eqnarray}
\phi_3^{\tn{mid}}(k,x,h) &=& \lambda_k\Bigl(x+\frac{h}{2}f(x)+\frac{h}{2}\sum_{j=1}^p \lambda_j(x)\,\nu_j\Bigl)\label{quad_midpoint}
\end{eqnarray}
and the \emph{trapezoidal quadrature rule}
\begin{eqnarray}
\phi_3^{\tn{trap}}(k,x,h) &=& \frac{1}{2}\,\lambda_k(x)+\frac{1}{2}\lambda_k\Bigl(x+h f(x)+h\sum_{j=1}^p \lambda_j(x)\,\nu_j\Bigl)\,.\label{quad_trapez}
\end{eqnarray}
\end{definition}

\begin{remark} The quadrature rules \eqref{quad_euler} -- \eqref{quad_trapez} may in general be combined arbitrarily with the scheme \eqref{def_theta_method} and any choice of $\theta$. However, as we will find in the error analysis, see Theorems \ref{theorem_theta_method_strong}, it is advantageous to combine deterministic higher order methods. That is, in the case $\theta=\tfrac{1}{2}$, which is an order two methods for ODEs, one should combine the $\Theta$-method \eqref{def_theta_method} with either the midpoint \eqref{quad_midpoint} or trapezoidal rule \eqref{quad_trapez} which are second order quadrature rules.
\end{remark}

\subsection{Analysis of scaled systems and relationship to small noise SDEs}\label{section_scaled_system}

In many applications an asymptotic error analysis with respect to the time-step may not, though mathematically sound, give a detailed enough picture in practical fixed time step situations. When the time step becomes asymptotically small one always enters the range wherein jump-adapted methods should be used and thus fixed time step methods will be employed for rather large time steps only. 
It is a central idea in the studies \cite{AndGanKurtz,AndersonHigham,AndersonKoyama} to introduce a scaling parameter in the underlying equations, which is often naturally given, e.g., as a system size parameter in chemical reaction systems. This scaling parameter is then also included in the error estimates and it naturally relates to the regime when tau-leaping methods are efficiently applicable.
This idea is closely related to small noise analysis of numerical methods which was carried out in the case of stochastic differential equations or jump diffusions \cite{BuckwarRiedlerJSDEs,BuRoWi,MilsteinSmallNoise}. In these studies it was found that in practice certain methods which are of the same order of convergence are particularly useful in different regimes of step-size to noise relations. 

In the following we use $N$ to denote the scaling parameter setting $X^N_i=N^{-\alpha_i}X_i$ to naturally scale the system, where usually $N\gg 1$. Here we choose $\alpha_i>0$ such that the resulting variables satisfy $X^N_i=\landau(1)$ for all $i=1,\ldots, d$. This yields the generic scaled RTE
\begin{equation}\label{scaled_random_time_change_eq}
X^N(t)=X^N(0)+N^\eta\int_0^t f^N(X^N(s))\,\dif s + \sum_{k=1}^p Y_k\biggl(N^\gamma \int_0^t N^{c_k} \lambda_k^N(X^N(s))\,\dif s\biggr)\nu^N_k\, ,
\end{equation}
where $\nu^N_k=\nu_k\cdot(N^{-\alpha_1},\ldots,N^{-\alpha_d})\in\rr^d$ and $\cdot$ denotes the entrywise product of two vectors. Further, the constants $\eta,\,\gamma$ and $c_k$ are also chosen such that the coefficients $f^n$, $\lambda_k^N$ are of order one. We note that it holds $|\nu^N_k|=\landau(N^{-c_k})$, where, however, $|\nu_k^N|\propto N^{-\rho_k}\ll N^{-c_k}$ is possible, i.e., $c_k\leq \rho_k$. Finally, we set $\rho=\min_{k=1,\ldots,p}\,\rho_k$. The notation here follows \cite{AndersonKoyama,AndersonHigham} for reasons of comparability as our results reduce to theirs in the case of $f\equiv0$, or formally $\eta=-\infty$. The appearance of the scaling in chemical reactions system in general and for particular examples is discussed in detail in \cite{AndGanKurtz,AndersonKoyama} and we do not want to repeat this discussion at present.

Here we comment on the relationship of the analysis of the scaled equation \eqref{scaled_random_time_change_eq} to small noise stochastic equations. In small noise equations the noise coefficients are characterised to be proportional to a parameter $\eps\ll 1$ measuring the `smallness' of the noise. The behaviour of the error of an approximation method is then characterised by different areas in the step size-to-noise parameter regime. For example, for step sizes $h\gg \eps$ the error in the drift approximation is dominant and deterministically higher order methods yield a smaller overall error. In the present setting smallness of the noise would translate to the product of the jump intensities and  the jump heights being of order $\eps\ll 1$, and the drift part being of order one, i.e., $\eta\approx 0$. Therefore, if it holds that $\gamma<0$ and hence $N^\gamma$ is actually small, by setting $\eps:=N^\gamma$ the scaled equation is analogous to a small noise equation and the analysis in the present study corresponds to small noise analysis. 
We note that the subsequent error estimates are most appropriate when these assumptions are satisfied. Considering as an application scaled hybrid chemical reaction systems, then the assumption $\gamma\leq 0$ is often satisfied \cite{AndersonHigham,AndersonKoyama}. Furthermore as in these models the drift part of the equation \eqref{scaled_random_time_change_eq} usually arises as taking reactions on a much faster time scale to their deterministic limit, i.e., the corresponding reaction rate equations, see \cite{Debussche} for a precise mathematical limit theorem, it is also reasonable to expect that this limit yields $N^\eta f=\landau(1)$, i.e., $\eta\approx 0$. For an example illustrating these discussions we refer to Section \ref{section_numerical_examples}.

\subsection{Main results}

We now state the convergence results for the $\Theta$-Maruyama method \eqref{def_theta_method} -- \eqref{quad_trapez}, the detailed and tedious calculations involved in the analysis of the local error are deferred to the subsequent section. We assume that the equation \eqref{random_time_change_eq} satisfies a scaling as detailed in Section \ref{section_scaled_system}, i.e., \eqref{random_time_change_eq} is of the form \eqref{scaled_random_time_change_eq}, and we use the scaling constants in the asymptotic expansion of the error. If one is interested in unscaled results simply set $N=1$. We repeat that the following results are meaningful in the sense that they contain the leading error terms under the step size restriction $h<(N^{\gamma}+N^\eta)^{-1}$ only. As detailed before it is mostly the case that $1\ll N$ and $\eta,\gamma\leq 0$ in which case this step-size condition is usually satisfied. Moreover, the stability constant depends exponentially on the Lipschitz constants of the increment functions, see Remark \ref{remark_stability_theorem}, which are presently given by the Lipschitz constants of the coefficients $f,\lambda_k$. Hence, the stability constant grows exponentially in $N^\gamma+N^\eta$. Thus, assuming $\gamma,\eta\leq 0$ also avoids dealing with absurdly large stability constants. Finally, we note that in the case of a pure jump process, i.e., $f\equiv 0$, our results reduce to the results presented in \cite{AndGanKurtz,AndersonHigham} and in the case of state-independent jump rates, i.e., $\lambda_k\equiv const.$, our results are analogous to the error analysis of jump-diffusions, see, e.g., \cite{PlatenLiberati}.

\begin{theorem}\label{theorem_theta_method_strong} The global strong error of the $\Theta$-Maruyama-method satisfies in combination with the Euler quadrature scheme
\begin{equation*}
\max_{n=1,\ldots,\nbar n}\EX|X(t_n)-\nhat X_n|\, =\, \landau\bigl(h^{1/2}\,N^{(\gamma-\rho)/2}(N^\gamma+N^\eta)+h\,(N^\gamma+N^\eta)^2\bigr).
\end{equation*}
The trapezoidal rule, i.e., $\theta=\tfrac{1}{2}$, in combination with a second order quadrature scheme satisfies
\begin{equation*}
\max_{n=1,\ldots,\nbar n}\EX|X(t_n)-\nhat X_n|\, =\, \landau\bigl(h^{1/2}\,N^{(\gamma-\rho)/2}(N^\gamma+N^\eta)+h\,N^{3\gamma}+h^{3/2}N^{(\gamma-\rho)/2}(N^\gamma+N^\eta)^2\bigr)\,.
\end{equation*}
\end{theorem}

\begin{proof} The statement in the theorem follows from Theorem \ref{convergence_theorem} and the strong local error estimates \eqref{strong_error_est_I} and \eqref{strong_error_est_II} derived in the subsequent section.
\end{proof}

\begin{remark} The improvement in the error due to a combination of two second order methods is that the term $hN^\eta$ is eliminated. Therefore, it is expected that the method performs better in the step-size range where the second order term is predominant and $N^\gamma$ is small compared to $N^\eta$. 
\end{remark}


\subsubsection{Improved trapezoidal and midpoint rule}\label{section_improved_method}

The intended gain in the use of deterministically higher order methods is that the first order term in the expansion of the strong error vanishes. This is only partly achieved using the `classical' midpoint and trapezoidal quadrature rules \eqref{quad_midpoint} and \eqref{quad_trapez} as a term $hN^{3\gamma}$ remains in the strong error expansion, see Theorem \ref{theorem_theta_method_strong}. However by a simple change in the definition of these methods, which reflects the discontinuous nature of the process, we obtain that also this term vanishes. To this end we define the \emph{improved midpoint quadrature rule}
\begin{eqnarray}
\phi_3^{\tn{imid}}(k,x,h) &=& \lambda_k\Bigl(x+\frac{h}{2}f(x)\Bigl)+\frac{h}{2}\sum_{j=1}^p \lambda_j(x)\bigl(\lambda_k(x+\nu_j)-\lambda_k(x)\bigr)\label{imp_quad_midpoint}
\end{eqnarray}
and the \emph{improved trapezoidal quadrature rule}
\begin{eqnarray}
\phi_3^{\tn{itrap}}(k,x,h) &=& \frac{1}{2}\,\lambda_k(x)+\frac{1}{2}\lambda_k\Bigl(x+h f(x)\Bigl)+\frac{h}{2}\sum_{j=1}^p \lambda_j(x)\bigl(\lambda_k(x+\nu_j)-\lambda_k(x)\bigr)\,.\phantom{xxxx}\label{imp_quad_trapez}
\end{eqnarray}

We state the strong error estimate for the improved methods in the subsequent theorem and comparing to Theorem \ref{theorem_theta_method_strong} we see that the first order term in the error expansion vanishes.

\begin{theorem}\label{theorem_imp_methods} The strong global error for the trapezoidal rule, i.e., $\theta=\frac{1}{2}$, in combination with the improved midpoint \eqref{imp_quad_midpoint} or improved trapezoidal quadrature rule \eqref{imp_quad_trapez} satisfies
\begin{equation*}
\max_{n=1,\ldots,\nbar n}\EX|X(t_n)-\nhat X_n|\, =\, \landau\bigl(h^{1/2}\,N^{(\gamma-\rho)/2}(N^\gamma+N^\eta)+h^{3/2}\,N^{(\gamma-\rho)/2}(N^\gamma+N^\eta)^2\bigr).
\end{equation*}
\end{theorem}

\begin{proof} The statement in the theorem follows from Theorem \ref{convergence_theorem} and the strong local error estimates \eqref{strong_error_est_III}, see Remark \ref{reamrk_improved_method}, and \eqref{strong_error_est_II} derived in the subsequent section.
\end{proof}

\section{Local error estimates for Theorems \ref{theorem_theta_method_strong} and  \ref{theorem_imp_methods}}\label{section_detailed_error_estimates}

In order to prove Theorems \ref{theorem_theta_method_strong} and \ref{theorem_imp_methods} we need to validate the conditions in Theorem \ref{convergence_theorem}, i.e., we have to establish the Lipschitz conditions \eqref{lipschitz_cond_1} and \eqref{lipschitz_cond_2} on the increment functions and perform a strong local error analysis. The first task is straightforward, as the Lipschitz conditions correspond to the stability conditions of well-analysed deterministic methods and follow immediately from Lipschitz conditions of the coefficients functions. Moreover, without loss of generality we might assume that the Lipschitz constants $L_1$ and $L_2$ are proportional to the Lipschitz constant on $f$ and $L_{3,k}$ is proportional to the Lipschitz constant on $\lambda_k$. It remains to perform the local error analysis, which is the content of the next sections.

\subsection{A first auxiliary analytical tool}

The following proposition collects asymptotic properties of the terms arising in the It{\^o}-formula \eqref{ito_formula}. In view of future estimates we now consider the scaled process $X^N$ introduced in Section \ref{section_scaled_system}. In order to obtain the corresponding unscaled results simply set $N=1$. 

\begin{lemma}\label{estimation_lemma} The generator \eqref{generator_pdmp} applied to a function $F\in C^1_b(\rr^d)$ satisfies
\begin{equation*}
\,\mathcal{A}F\,=\, \landau\bigl(L_F\,(N^\eta+N^\gamma)\bigr)
\end{equation*}
and the martingale $(M^F(t))_{t\geq 0}$ arising in the It\^o-formula \eqref{ito_formula} applied to $F$ satisfies for all $t,h\geq 0$ that
\begin{equation*}
\EX|M^F(t+h)-M^F(t)|\,=\,\landau\Bigl(h^{1/2}\,L_F\,N^{\gamma/2}N^{-\rho/2}\Bigr)\,,
\end{equation*}
where $\rho=\min_{k=1,\ldots,p}\,\rho_k$.
\end{lemma}

\begin{proof} The generator \eqref{generator_pdmp} applied to a function $F$ gives
\begin{eqnarray*}
\big|\mathcal{A}F(x)\big|
&=&\Big|\nabla F[f](x) +\sum_{k=1}^p\lambda_k(x)\,\bigl(F(x+\nu_k)-F(x)\bigr)\Big|\\[2ex]
&\leq&\|\nabla F\|_0\Bigl(\|f\|_0 +\sum_{k=1}^p\|\lambda_k\|_0\,|\nu_k|\Bigr).
\end{eqnarray*}
The first result in the lemma follows considering the scaled process, i.e., $f$ is substituted by $N^\eta f^N$, $\lambda_k$ is substituted by $N^{\gamma+c_k}\lambda_k^N$ and $\nu_k$ is substituted by $\nu_k^N$ which satisfies $|\nu_k^N|=\landau(N^{-\rho_k})$, where $c_k\leq\rho_k$. Thus it remains to consider the asymptotic behaviour of the martingale $M^F$. Jensen's inequality implies $\EX|M^F(t+h)-M^F(t)|\leq\bigl(\EX|M^F(t+h)-M^F(t)|^2\bigr)^{1/2}$ and hence estimating the integrand in the right hand side of the quadratic variation \eqref{def_quad_variation} yields almost surely
\begin{equation*}
\sum_{k=1}^p\lambda_k(X(s))\,\big|F(X(s)+\nu_k)-F(X(s))\big|^2\leq \|\nabla F\|_0^2\,\sum_{k=1}^p\|\lambda_k\|_0\,|\nu_k|^2 = \landau\bigl(L_F^2\,N^{\gamma-\rho}\bigl),
\end{equation*}
and the second result in the lemma follows.
\end{proof}

\medskip

We now present the local error analysis of the $\Theta$-Maruyama-methods. The strong local error estimates \eqref{strong_error_est_I}, \eqref{strong_error_est_III} and \eqref{strong_error_est_II} imply in combination with the convergence theorem the results in Theorem \ref{theorem_theta_method_strong} and Theorem \ref{theorem_imp_methods}.

\subsection{The strong local error -- the terms $K_{n+1}$}\label{section_first_strong_error}

We need to consider the modulus of the one-step approximations 
\begin{equation*}
\EX|K_{n+1}^\tn{method}|\,=\, \EX\Big|\sum_{k=1}^p\Bigl(\int_{t_n}^{t_n+h}\lambda_k(X(s))\,\dif s-h\phi_3^{\tn{method}}(k,X(t_n),h)\Bigr)\,\nu_k\Big|
\end{equation*}
and we prove in this section that
\begin{equation}\label{strong_error_est_I}
\left.\begin{array}{rcl}
     \EX |K_{n+1}^\tn{euler}|&=&\landau\bigl(h^{3/2}\,N^{3\gamma/2-\rho/2}+h^2\,N^{\gamma}(N^\gamma+N^\eta)\bigr),\\[2.5ex]  
     \EX |K_{n+1}^\tn{mid/trap}|&=& \landau\bigl(h^{3/2}N^{3\gamma/2-\rho/2}+h^2N^{3\gamma}\bigl)+ \landau\bigl( h^{5/2}N^{3\gamma/2-\rho/2}(N^\gamma+N^\eta)\bigl)\\[2.5ex]
&&\mbox{} +\landau\bigl(h^3N^{\gamma}(N^{\gamma}+N^{\eta})^2+h^3N^{2\gamma}(N^{\gamma}+N^{\eta})^2\bigr).
\end{array}\right.
\end{equation}
In order to obtain the error estimates \eqref{strong_error_est_I} we analyse the error $\EX|K_{n+1}^\tn{method}|$ using (It{\^o}-)Taylor expansions of the involved functions and for simplicity of notation we omit the index $n$. First of all, an It{\^o}-Taylor-expansion of the integrand in the term correspoding to the exact solution yields
\begin{equation}\label{lambda_ito_1}
\int_t^{t+h}\lambda_k(X(s))\,\dif s= h\,\lambda_k(X(t))+\int_t^{t+h}\int_t^{s}\mathcal{A}\lambda_k(X(r))\,\dif r\,\dif s + \int_t^{t+h} \Bigl(M^{\lambda_k}(s)-M^{\lambda_k}(t)\Bigr)\,\dif s\,,
\end{equation}
where the martingal term satisfies due to Lemma \ref{estimation_lemma} and as $\|\nabla\lambda_k^N\|=\landau(N^{\gamma+c_k})$ that
\begin{equation*}
\EX\big|M^{\lambda_k}(s)-M^{\lambda_k}(t)\big|\ =\ \landau\bigl((s-t)^{1/2}N^{3\gamma/2+c_k-\rho/2}\bigr)\,.
\end{equation*}
We are now able to estimate the error in the case of the Euler quadrature rule. Subtracting the increment function $h\phi_3^{\tn{euler}}(k,X(t),h)\,=\,h\lambda_k(X(t))$ from the expansion \eqref{lambda_ito_1} we obtain
\begin{equation*}
\EX\Big|\int_{t}^{t+h}\lambda_k(X(s))\,\dif s-h\phi_3^{\tn{euler}}(k,X(t),h)\Big|\ = \ \landau\bigl(h^{3/2}\,N^{3\gamma/2+c_k-\rho/2}+h^2\,N^{\gamma+c_k}(N^\gamma+N^\eta)\bigr)\,.
\end{equation*}
The triangle inequality and $|\nu_k|=\landau(N^{-\rho_k})$ with $c_k\leq \rho_k$ implies the estimate for $\EX |K_{n+1}^\tn{euler}|$ in \eqref{strong_error_est_I}.

\medskip
Some more work is needed to estimate the local error for the midpoint or trapezoidal quadrature rule. A Taylor expansion of the respective increment functions yields in the case of the midpoint rule
\begin{eqnarray*}
h\phi_3^{\tn{mid}}(k,X(t),h) &=& h\,\lambda_k\Bigl(X(t)+\frac{h}{2}f(X(t))+\frac{h}{2}\sum_{j=1}^p \lambda_j(X(t))\nu_j\Bigl)\\
&=& h\,\lambda_k(X(t))+\frac{h^2}{2}\nabla\lambda_k[f](X(t))+\frac{h^2}{2} \nabla\lambda_k\Bigl[\sum_{j=1}^p \lambda_j(\cdot)\nu_j\Bigr](X(t)) + \varrho_1,
\end{eqnarray*}
where the remainder term satisfies $\varrho_1=\landau\bigl(h^3N^{\gamma+c_k}(N^{\gamma}+N^{\eta})^2\bigr)$, and for the trapezoidal rule
\begin{eqnarray*}
h\phi_3^{\tn{trap}}(k,X(t),h) &=& \frac{h}{2}\,\lambda_k(X(t))+\frac{h}{2}\lambda_k\Bigl(X(t)+h f(X(t))+h\sum_{j=1}^p \lambda_j(X(t))\nu_j\Bigl)\\
&=& h\,\lambda_k(X(t))+\frac{h^2}{2}\nabla\lambda_k[f](X(t))+\frac{h^2}{2}\nabla\lambda_k\Bigl[\sum_{j=1}^p \lambda_j(\cdot)\nu_j\Bigr](X(t)) + \varrho_2,
\end{eqnarray*}
where the remainder term $\varrho_2$ possesses the same asymptotics as $\varrho_1$. 
\medskip

In order to compare this expansion to $\lambda_k(X(s))$ we further expand in the right hand side of \eqref{lambda_ito_1} the two summands in the generator $\mathcal{A}\lambda_k$ using the It\^o-formula \eqref{ito_formula}. First, for the continuous part this yields
\begin{equation*}
\nabla\lambda_k[f](X(s))=\nabla\lambda_k[f](X(t))+\int_t^{s}\mathcal{A}\bigr(\nabla\lambda_k[f]\bigl)(X(r))\,\dif r+\Bigl(M^{\nabla\lambda_k[f]}(s)-M^{\nabla\lambda_k[f]}(t)\Bigr)
\end{equation*}
and for the jump part
\begin{eqnarray*}
\sum_{j=1}^p\lambda_j(X(s))\,\Bigl(\lambda_k(X(s)+\nu_j)-\lambda_k(X(s))\Bigr)
&=&\sum_{j=1}^p\lambda_j(X(t))\,\Bigl(\lambda_k(X(t)+\nu_j)-\lambda_k(X(t))\Bigr)\\[2ex]
&&\mbox{} + \int_t^{s}\mathcal{A}\Bigl(\sum_{j=1}^p\lambda_j(\cdot)\,\Bigl(\lambda_k(\cdot+\nu_j)-\lambda_k(\cdot)\Bigr)(X(r))\,\dif r\\[2ex]
&&\mbox{} + M^{\nabla\lambda_k[\lambda_j\nu_j]}(s)-M^{\nabla\lambda_k[\lambda_j\nu_j]}(t)\,.
\end{eqnarray*}
Finally, we further expand the first term on the right hand side using Taylor's Theorem and obtain due to linearity of the gradient vector
\begin{eqnarray*}
\sum_{j=1}^p\lambda_j(X(t))\,\Bigl(\lambda_k(X(t)+\nu_j)-\lambda_k(X(t))\Bigr)
&=&\nabla\lambda_k\Bigl[\sum_{j=1}^p\lambda_j(\cdot)\nu_j\Bigr](X(t)) +\, \varrho_3\, 
\end{eqnarray*}
where the remainder term satisfies $\varrho_3=\landau(N^{3\gamma+c_k})$. 
Thus, overall we obtain the expansion
\begin{eqnarray*}
\lefteqn{\int_t^{t+h}\lambda_k(X(s))\,\dif s\ =}\\
&=& h\,\lambda_k(X(t))
+ \frac{h^2}{2}\,\nabla\lambda_k[f](X(t)) 
+ \frac{h^2}{2}\,\nabla\lambda_k\Bigl[\sum_{j=1}^m\lambda_j(\cdot)\nu_j\Bigr](X(t))+ \frac{h^2}{2}\,\varrho_3 \\[2ex]
&&\hspace{-20pt}\mbox{} +\int_t^{t+h}\int_t^{s}\int_t^{r}\mathcal{A}\Bigl(\sum_{j=1}^p\lambda_j(\cdot)\,\Bigl(\lambda_k(\cdot+\nu_j)-\lambda_k(\cdot)\Bigr)(X(v))\,\dif v\,\dif r\,\dif s + \int_t^{t+h} \Bigl(M^{\lambda_k}(s)-M^{\lambda_k}(t)\Bigr)\,\dif s\\[2ex]
&&\hspace{-20pt}\mbox{} + \int_t^{t+h}\int_t^{s} \Bigl(M^{\nabla\lambda_k[f]}(r)-M^{\nabla\lambda_k[f]}(t)\Bigr)+\Bigl(M^{\nabla\lambda_k[\lambda_j\nu_j]}(r)-M^{\nabla\lambda_k[\lambda_j\nu_j]}(t)\Bigr)\,\dif r\,\dif s\,.
\end{eqnarray*}
Here due to Lemma \ref{estimation_lemma} the two last martingale terms in the right hand side possess the following asymptotics
\begin{eqnarray*}
\EX\big|M^{\nabla\lambda_k[f]}(s)-M^{\nabla\lambda_k[f]}(t)\big|&=&\landau\bigl((s-t)^{1/2}N^{3\gamma/2+c_k-\rho/2}\,N^\eta\bigr)\,,\\[2ex]
\EX\big|M^{\nabla\lambda_k[\lambda_j\nu_j]}(s)-M^{\nabla\lambda_k[\lambda_j\nu_j]}(t)\big|&=&\landau\bigl((s-t)^{1/2}N^{5\gamma/2+c_k-\rho/2}\bigr)\,.
\end{eqnarray*}
Therefore, subtracting the two expansions we obtain the error estimate
\begin{eqnarray*}\label{local_poisson_strong_mid_error}
\EX\Big|\int_{t}^{t+h}\lambda_k(X(s))\,\dif s-h\phi_3^{\tn{mid/trap}}(k,X(t),h)\Big|& =& \landau\bigl(h^{3/2}N^{3\gamma/2+c_k-\rho/2}+h^2N^{3\gamma+c_k}\bigl)\nonumber\\[2ex]
&&\hspace{-1cm}\mbox{} + \landau\bigl( h^{5/2}N^{3\gamma/2+c_k-\rho/2}(N^\gamma+N^\eta)\bigl)\\[2.5ex]
&&\hspace{-1cm}\mbox{} +\landau\bigl(h^3N^{\gamma+c_k}(N^{\gamma}+N^{\eta})^2+h^3N^{2\gamma+c_k}(N^{\gamma}+N^{\eta})^2\bigr)\nonumber
\end{eqnarray*}
and the estimate for $\EX |K_{n+1}^\tn{mid/trap}|$ in \eqref{strong_error_est_I} follows.

\begin{remark}\label{reamrk_improved_method} The proof is completely analogous for the improved methods in Section \ref{section_improved_method}, with the following two exceptions. First, the Taylor expansion of the increment function yields
\begin{equation*}
h\phi_3^\tn{imid/itrap}(k,X(t),h)= h\,\lambda_k(X(t))+\frac{h^2}{2}\nabla\lambda_k[f](X(t))+\sum_{j=1}^p\lambda_j(X(t))\,\Bigl(\lambda_k(X(t)+\nu_j)-\lambda_k(X(t))\Bigr)+\varrho_4,
\end{equation*}
where the remainder term satisfies $\varrho_4=\landau(N^{\gamma+c_k}N^{2\eta})$. Secondly, this expansion yields that the remainder term $\varrho_3$ vanishes. Hence, we obtain the improved estimate
\begin{equation}\label{strong_error_est_III}
\EX|K_{n+1}^{\tn{imid/itrap}}|\,=\,\landau\bigl(h^{3/2}N^{3\gamma/2-\rho/2}+h^{5/2}N^{3\gamma/2-\rho/2}(N^\gamma+N^\eta)\bigr)
\end{equation}
which does not contain a second order term.
\end{remark}

\subsection{The strong local error -- the terms $L_{n+1}$}

In this section we consider the local error of the drift increment function, i.e.,
\begin{equation*}
\EX|L_{n+1}^\tn{theta}| \ =\  \EX\Big|\int_{t_n}^{t_{n}+h} f(X(s))\,\dif s - h(1-\theta)f(X(t_n))-h\theta f(X(t_n+h))\Big|
\end{equation*}
and prove the local error estimates
\begin{equation}\label{strong_error_est_II}
\EX |L_{n+1}^\tn{theta}|\ =\
\left\{\begin{array}{cl}
\landau\bigl(h^{3/2}\,N^{\gamma/2-\rho/2}N^\eta + h^2\,N^\eta(N^\gamma+N^\eta)\bigr) & \tn{if } \theta\neq\tfrac{1}{2}\,,\\[2ex]
\landau\bigl(h^{3/2}\,N^{\gamma/2-\rho/2}N^\eta+h^{5/2}\,N^{\gamma/2-\rho/2}N^\eta(N^\gamma+N^\eta)\bigr) & \tn{if } \theta=\tfrac{1}{2}\,.
\end{array} \right.
\end{equation}

Analogously to the preceding section we use (It\^o-)Taylor expansions of the summands in the definition of the local error in order to obtain the estimates \eqref{strong_error_est_II}. Again we omit the index $n$. First, expanding the term $f(X(s))$ using the It{\^o} formula \eqref{ito_formula} yields
\begin{eqnarray*}
\int_t^{t+h}f(X(s))\,\dif s&=& h\,f(X(t))+\int_t^{t+h}\int_t^{s} \mathcal{A}f(X(r))\,\dif r\,\dif s + \int_t^{t+h}\Bigl(M^f(s)-M^f(t)\Bigr)\,\dif s
\end{eqnarray*}
and
\begin{eqnarray*}
f(X(t+h))&=& f(X(t))+\int_t^{t+h} \mathcal{A}f(X(s))\,\dif s + \Bigl(M^f(t+h)-M^f(t)\Bigr)\,.
\end{eqnarray*}
Therefore, the local error of the $\Theta$-Maruyama method becomes
\begin{eqnarray*}
\lefteqn{\EX\Big|\int_{t}^{t+h} f(X(s))\,\dif s - h(1-\theta) f(X(t))-h\theta f(X(t+h))\Big|}\\[2ex]
&&\phantom{xxxxxxxxxxxxxxxxxxxxx}\leq\ \EX\Big|\int_t^{t+h} \Bigl(M^f(s)-M^f(t)\Bigr)\,\dif s\Big|+h\theta\,\EX\Big|M^f(t+h)-M^f(t)\Big|
 \\&&\phantom{xxxxxxxxxxxxxxxxxxxxx\leq}\mbox{}+\EX\Big|\int_t^{t+h}\int_t^{s}\mathcal{A}f(X(r))\,\dif r\,\dif s- h\theta\int_t^{t+h} \mathcal{A}f(X(s))\,\dif s\Big|\,.
 \end{eqnarray*}
The individual terms in the right hand side are now estimated using Lemma \ref{estimation_lemma} and for the last term we note that it satisfies
\begin{equation*}
\EX\Big|\int_t^{t+h}\int_t^{t'}\mathcal{A}f(X(s))\,\dif s\,\dif t'- h\theta\int_t^{t+h} \mathcal{A}f(X(s))\,\dif s\Big|\ =\ \landau(h^2N^\eta(N^\gamma+N^\eta)).
\end{equation*}
If $\theta=1/2$ a further application of the It\^o-formula shows that the order two terms cancel and we obtain
\begin{equation*}
\EX\Big|\int_t^{t+h}\int_t^{s}\mathcal{A}f(X(r))\,\dif r\,\dif s- \frac{h}{2}\int_t^{t+h} \mathcal{A}f(X(s))\,\dif s\Big|\ =\ \landau(h^{5/2}N^{\gamma/2+\rho/2}N^\eta(N^\gamma+N^\eta))\,.
\end{equation*}
Overall, this establishes the estimates \eqref{strong_error_est_II}.

\section{Numerical examples}\label{section_numerical_examples}

In this section we provide two numerical examples to illustrate the analytical results for the $\Theta$-Maruyama methods. The first example is a simple linear, scalar model for which an exact solution can be calculated and the second is a phenomenological model of bacteriophage growth. 

For the convergence experiments we compared approximations obtained by $\Theta$-methods calculated from the same Poisson paths and the endpoint and estimated the strong error by
\begin{equation*}
\EX|X(T)-\nhat X_{Nh}|\approx\frac{1}{M}\sum_{j=1}^M |X(T,\omega_j)-\nhat X_{Nh}(\omega_j)|
\end{equation*}
for $M=200$ trajectories. 

\subsection{Example 1 - linear, scalar RTE}

As a first example we consider the linear, scalar RTE
\begin{equation}\label{numerical_example_1}
X(t) = X(0) -\alpha \int_0^t X(s)ds + \epsilon Y\biggl(\lambda\int_0^t X(s) \dif s \biggr)
\end{equation}
with positive coefficients $\alpha,\lambda,\eps>0$ on the time interval $[0,T]=[0,5]$ started at $X(0)=10$. The linearity of the coefficient functions allows to analytically solve the inter-jump ODE and the integral along the inter-jump motion. Therefore an exact solution to \eqref{numerical_example_1} can be constructed for a given path of the Poisson process $Y$. We have carried out convergence experiments for two sets of coefficients. In the first set the noise strength $\eps\lambda$ is of approximately the same order as the drift whereas in the second set the noise strength is reduces and thus corresponds to small noise. We implemented the explicit and implicit Euler method combined with an Euler quadrature rule and the trapezoidal rule combined with a second order quadrature rule. Note that for this example the trapezoidal, the midpoint and the improved methods all coincide. The plots in Figure \ref{example_1_strong} report the results for the strong error of these methods. For the first set of parameters one can observe in the left plot the predicted convergence behaviour of order $1/2$. However, note that in comparison to the Euler method the trapezoidal rule shows a better convergence behaviour. In comparison in the plot on the right an improved convergence rate of $1$ can be seen for the small noise case.

\begin{figure}[htp]
\begin{center}
\includegraphics[width=0.49\textwidth, clip=true, trim= 0 0 0 0]{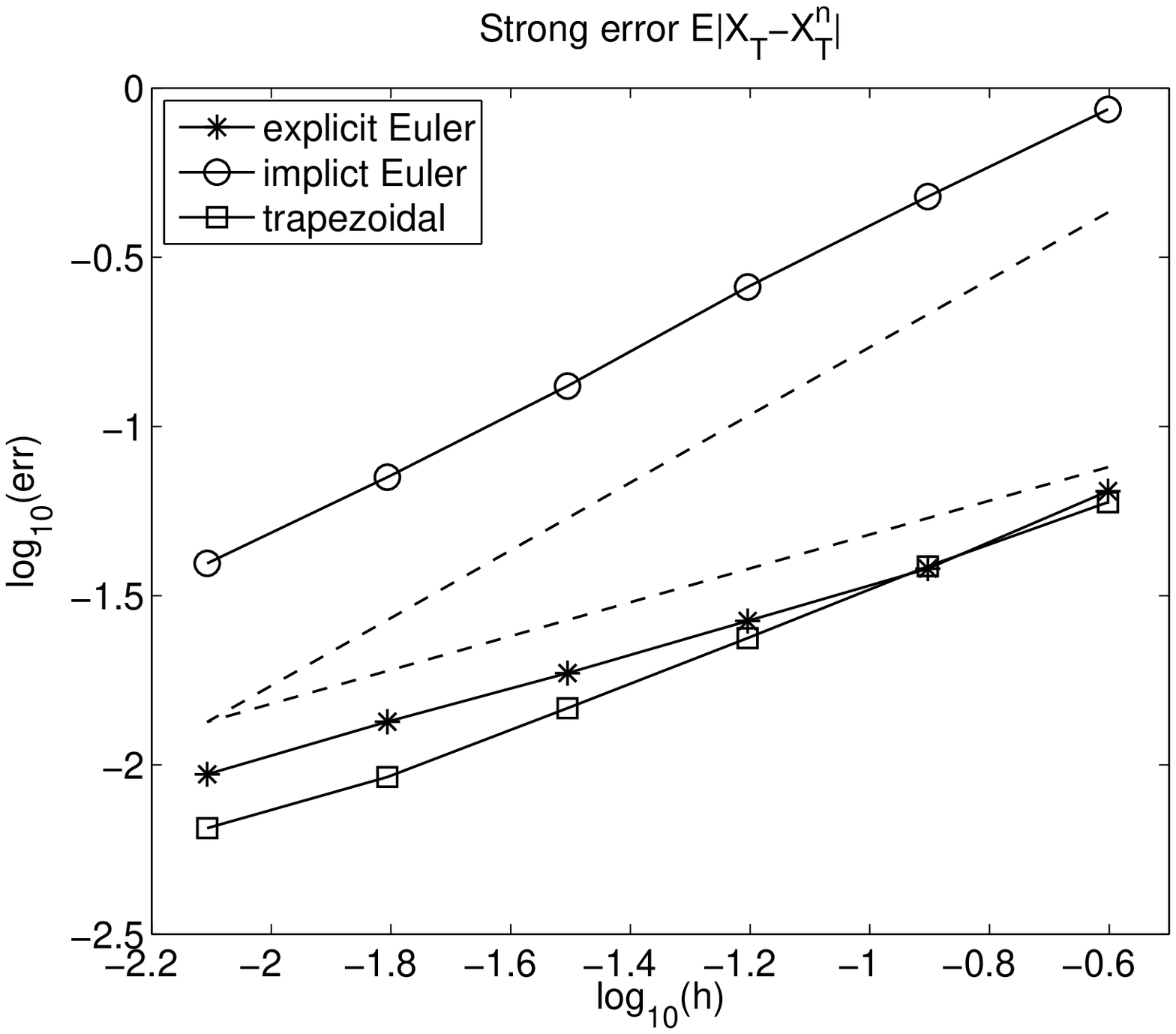}
\includegraphics[width=0.49\textwidth, clip=true, trim= 0 0 0 0]{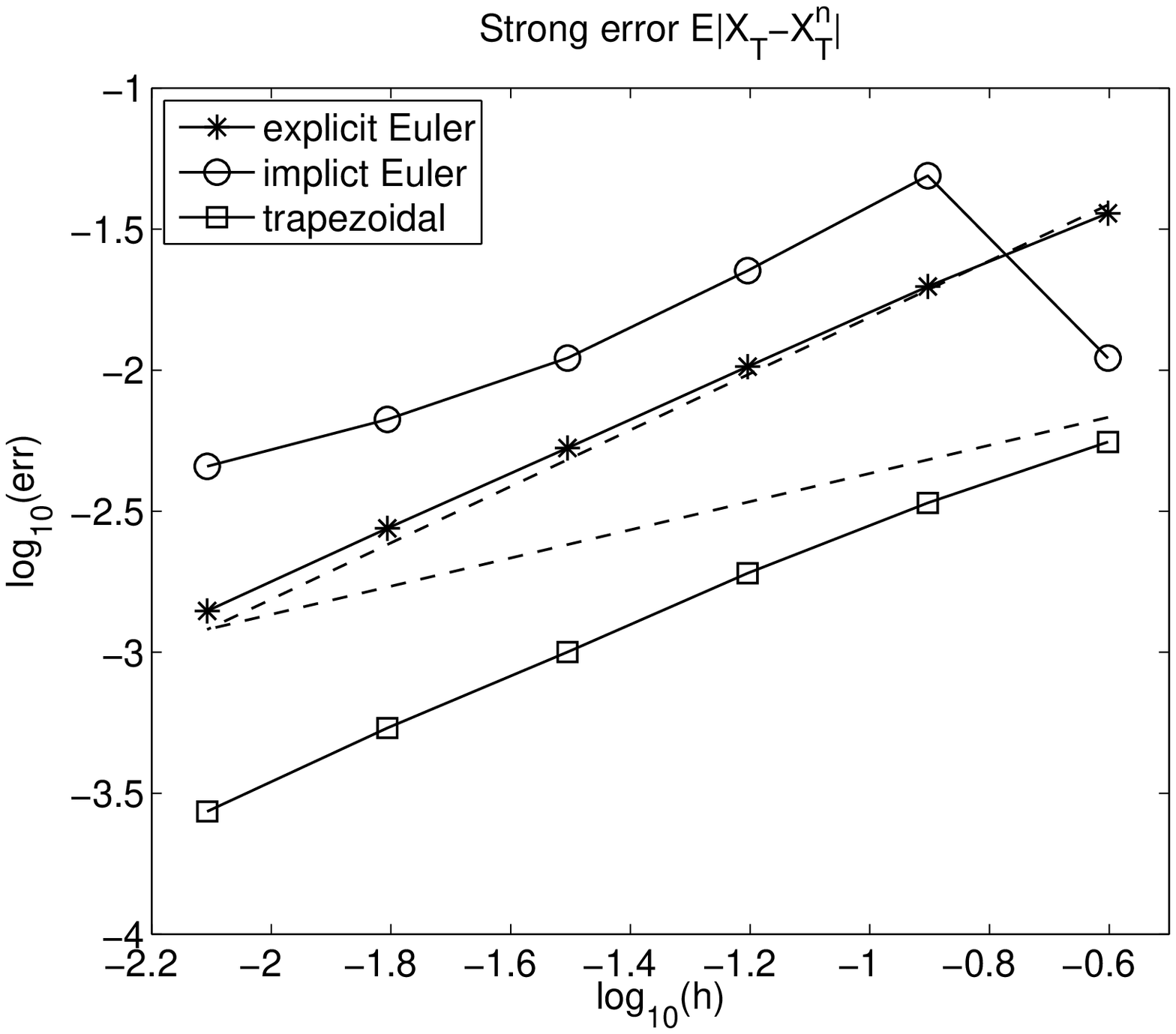}
\end{center}\caption{Convergence plots for the strong error for Example 1 with the parameter $\alpha=1.5$, $\lambda=200$, $\eps=0.007$ (left) and $\alpha=1.5$, $\lambda=500$, $\eps=0.001$ (right). Dashed lines are order lines with slopes $0.5$ and $1$ for comparison.}\label{example_1_strong}
\end{figure}

\subsection{Example 2 - bacteriophage model} 

As a second example we consider a chemical reaction system which models the growth of a bacteriophage inside a cell. This schematic model was developed in \cite{Srivastava2002} and is repeatedly considered as an example to illustrate the accuracy and computational efficiency of approximate simulations, most recently in \cite{AndersonHigham} in the context of Multilevel Monte Carlo for tau leaping methods. The model is based on a simple network, shown in Figure \ref{network}, which describes six reactions between the viral nucleic acids, classified as genomic (\emph{gen}) or 
template (\emph{tem}), and a viral structural protein (\emph{struc}). For a detailed description of this network and the additional assumptions made when modeling, we refer to \cite{Srivastava2002}. The reaction rates, expressed in day$^{-1}$, are 
\begin{equation*}
\left.\begin{array}{rclrclrcl}
r_1 &=& 0.025, & r_2 &=& 0.25, & r_3 &=& 1,\\[1ex]
r_4 &=& 7.5\cdot10^{-6}, & r_5&=& 1000, & r_6 &=& 1.9985.
\end{array}\right.
\end{equation*}
\begin{figure}
\begin{minipage}{0.65\textwidth}
\begin{center}
\includegraphics[width=0.7\textwidth,clip=true,trim=0 0 0 0]{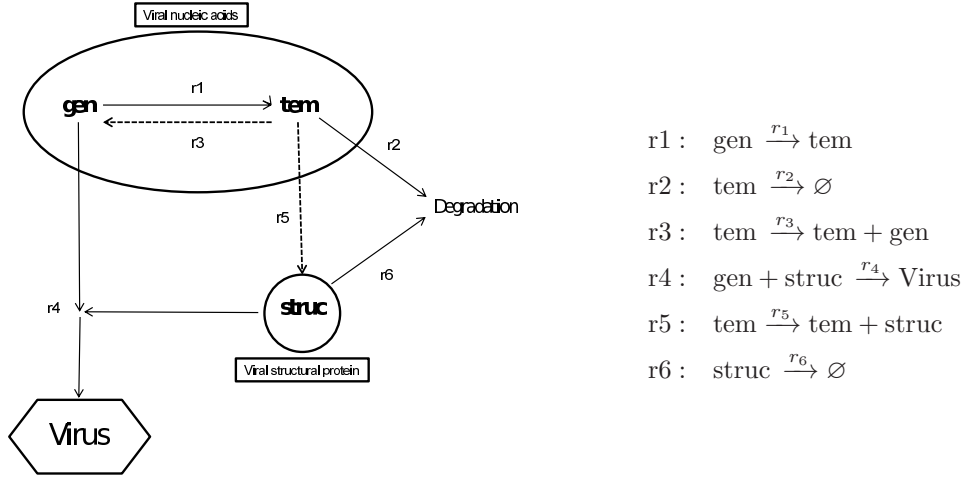}
\end{center}
\end{minipage}
\begin{minipage}{.30\textwidth}
 \begin{align*}
  \text{r1} :& \quad \text{gen} \, \xrightarrow{r_1}\text{tem} \\ 
  \text{r2} :& \quad \text{tem} \, \xrightarrow{r_2}  \varnothing \\
  \text{r3} :& \quad \text{tem} \, \xrightarrow{r_3}  \text{tem} + \text{gen} \\
  \text{r4} :& \quad \text{gen} + \text{struc} \, \xrightarrow{r_4} \text{Virus} \\
  \text{r5} :& \quad \text{tem}  \xrightarrow{r_5}  \text{tem} + \text{struc} \\
  \text{r6} :& \quad \text{struc} \, \xrightarrow{r_6}  \varnothing
 \end{align*}
\end{minipage}
\caption{Reaction diagram between the three components of a viral replication cycle: genome (\textbf{gen}), template (\textbf{tem}) and structural (\textbf{struc}). The two reactions $r3$ and $r5$, represented by ($\mathbb{---}$) lines, are catalytic and use the component \textbf{tem} for the synthesis of the other two components.}\label{network}
\end{figure}

This model is interesting to analyse for two reasons. Firstly the three components vary on different orders of magnitude: for the same time scale, the range of \emph{struc} oscillates between $10^2$ to $10^3$ molecules, whereas the fluctuation of \emph{ten} and \emph{gen} is not more than $10$ molecules. Secondly the model presents a bimodal distribution: a first peak can be achived when all the species become populated (``typical'' infection), while a second peak can be generated if all the species are eliminated from the cell (``aborted'' infection). It was shown that both features are better described by a stochastic model than a deterministic one  \cite{Haseltine}. A hybrid version was introduced and considered in \cite{Haseltine, Kalantzis, Huisinga}, based on the observation that, close to the equilibrium point of the corresponding deterministic system $(20,200,10\cdot 10^3)$, the rates of reactions $r5$ and $r6$ are three orders of magnitude higher than those of the reactions $r1$ to $r4$, i.e., $5,5,20$ and $15$ compared with $20\cdot10^3$ and $19\cdot10^3$. It is then chosen to model the reactions $r5$ and $r6$ deterministically while the others are modelled stochastically. The proposed hybrid model thus takes the form 
\begin{align}\label{hybrid}
 X_1(t) = & X_1(0) + Y_1\biggl(\int_0^t r_1X_2(s)ds\biggr) - Y_2\biggl(\int_0^t r_2X_1(s)ds\biggr), \nonumber  \\
X_2(t) = & X_2(0) - Y_1\biggl(\int_0^t r_1X_2(s)ds\biggr) + Y_3\biggl(\int_0^t r_3X_1(s)ds\biggr) -  Y_4\biggl(\int_0^t r_4X_2(s)X_3(s)ds\biggr),   \\
X_3(t) = & X_3(0) + \int_0^t (r_5X_1(s) -  r_6X_3(s) ) ds  -  Y_4\biggl(\int_0^t r_4X_2(s)X_3(s)ds\biggr),  \nonumber
\end{align}
where $X(t)=(X_1(t),X_2(t),X_3(t))=(tem,gen,struc)$ denotes the number of molecules in the system at time $t \in [0,T]$ and $Y_1$,\ldots, $Y_4$ are independent unit rate Poisson processes.

We next rescale the equations \eqref{hybrid}, cf.~Section \ref{section_scaled_system}. For the convergence experiments the scaled system possesses the advantage that the errors of all components of the resulting process $X^N(t)$ are of the same order of magnitude, however all dynamics of the process are exactly preserved. We rescale the model \eqref{hybrid} by defining $X_i^N = N^{-\alpha_i}X_i$ and and thus by substitution and rewriting the model in vectorial form we obtain 
\begin{eqnarray*}
X^N(t) &=& X^N(0) + N^\eta\int_0^t\left(\!\!\begin{array}{ccc}
 0 & 0 & 0 \\
 0 & 0 & 0 \\
 N^{c_5}r_5 & 0 & -N^{c_6}r_6
\end{array}\!\!\right)\, X^N(s)\,\dif s\\ &&\mbox{}
+ Y_1\biggl(N^\gamma\int_0^t N^{c_1}r_1X_2^N(s)\,\dif s\biggr)\,\left(\!\!\begin{array}{c} N^{-\alpha_1}\\ -N^{-\alpha_2}\\ 0 \end{array}\!\!\right)
+ Y_2\biggl(N^\gamma\int_0^t N^{c_2}r_2X_1^N(s)\,\dif s\biggr)\,\left(\!\!\begin{array}{c} -N^{-\alpha_1}\\ 0\\ 0 \end{array}\!\!\right)\\
&&\mbox{}
+ Y_3\biggl(N^\gamma\int_0^t N^{c_3}r_3X_1^N(s)\,\dif s\biggr)\,\left(\!\!\begin{array}{c} 0\\ N^{-\alpha_2}\\ 0 \end{array}\!\!\right)
+ Y_4\biggl(N^\gamma\int_0^t N^{c_4}r_4X_2^N(s)X^N_3(s)\,\dif s\biggr)\,\left(\!\!\begin{array}{c} 0\\ -N^{-\alpha_2}\\ -N^{-\alpha_3} \end{array}\!\!\right).
\end{eqnarray*}
Here we have chosen the rescaling coefficients as follows: $N=10000$, $\alpha=(1/4,1/2,1)$, for the remaining exponents $(c_1,\ldots,c_6)=(1/2,1/4,1/4,3/2,-3/4,0)$ and thus we obtain for the noise intensities $\eta,\gamma=0$.

For the convergence experiment using the scaled system we set the initial $X^N(0)=(2, 2, 1)$ which is the equilibrium of the corresponding deterministic reaction rate model after scaling and we have simulated the paths over the interval $[0,10]$. We implemented the explicit and implicit Euler method combined with an Euler quadrature rule, the trapezoidal rule combined with a the trapezoidal and with the midpoint quadrature rule as well as with their improved versions. During the simulations, we occasionally encountered the problem of the presence of negative populations. In the literature this problem is well known and recently several solutions were proposed to avoid this unacceptable occurrence \cite{Gillespie2005, BurrageTian2004}. In order to resolve this problem we reset the value of any negative component to zero after any step. As an exact solution cannot be constructed for this example we used a numerical reference solution with a the small stepsize of $h=\frac{1}{320}$. The results are reported in Figure \ref{example_2_strong} where again we can observe a convergence behaviour along the theoretical findings.

\begin{figure}[htp]
\begin{center}
\includegraphics[width=0.75\textwidth, clip=true, trim= 30 25 40 10]{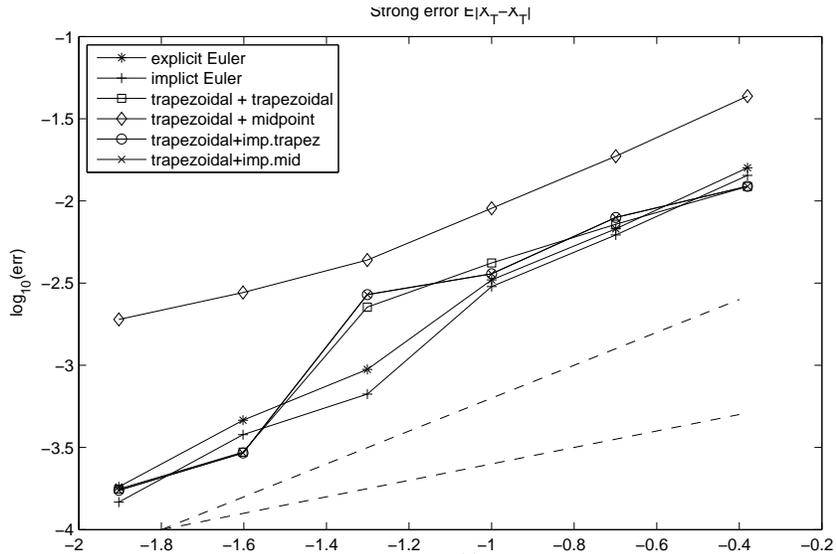}
\end{center}\caption{Convergence plots of the strong error for the scaled bacteriophage model. The dashed lines are  order lines with slope $0.5$ and $1$ for comparison.}\label{example_2_strong}
\end{figure}

\section{Conclusions}\label{section_conclusions}

In this study we have first established the equivalence between RTEs of the form \eqref{random_time_change_eq} and certain PDMPs which arise, e.g., in biochemical applications. Secondly, we have developed a convergence analysis for fixed time step approximation methods for RTEs in the strong sense. The development of a convergence analysis for numerical methods for RTEs which has not been possible for PDMPs shows that it is advantageous to write PDMPs as RTEs when possible.
For the numerical analysis we have restricted ourselves to equations with sufficiently smooth and bounded autonomous coefficients and equidistant approximation grids. Extensions to grids with variable step sizes and to equations with time-dependent coefficients are trivial generalisations of our main convergence result in Theorem \ref{convergence_theorem}. In the former case one simply has to chose the maximal step size as the parameter for the convergence. In the latter case of non-homogeneous equations \eqref{random_time_change_eq}, i.e., the coefficients and rates further depend explicitly on time, passing to the space-time process essentially reduces the problem to the current setting and for the numerical solution we simply choose the suitable time-dependent methods. 
We further conjecture that combining the present analysis with suitable methods employed already in the analysis of stochastic differential equations might provide extensions of the convergence results for less restrictive assumptions on the coefficients, e.g., unbounded coefficients or coefficients satisfying only a local or a one-sided Lipschitz condition.
This conjecture is also supported by the second numerical example we presented wherein the conditions on the coefficients are actually violated. In general the numerical examples illustrated the theoretical findings that using deterministically second order methods yielded an improvement over first order methods in regimes when the noise is small compared to the deterministic part of the equation. We conjecture comparing to results for jump diffusions, see \cite{BuckwarRiedlerJSDEs} that these observations might be even pronounced when using methods of even higher deterministic order, e.g., Runge-Kutta methods for the drift part. We note that our general convergence theorem can be used to analyse the error also for this type of methods. One restriction, however, of extending the presented method for general PDMPs is that due to the equivalence to RTEs of the form \eqref{random_time_change_eq} the method is restricted to processes with a finite number of possible jump heights. We plan to address this important issue in future work as we are convinced that for high frequency jumps fixed time step methods will be more efficient and thus the best choice in practice.

We close the study with a brief discussion of an alternative approach to approximating jump processes with state-dependent intensities. A weak jump-adapted Euler-Maruyama method for jump-diffusion with state-dependent intensities was introduced in \cite{GlaMer2}. If considered without a Brownian motion part solutions to these equations are again versions of PDMPs and thus, if the jump structure is appropriate, also of RTEs. The central idea for the construction of the approximation method in \cite{GlaMer2} is a thinning procedure. If the total jump intensity is bounded, then such a process is further equivalent to a process driven by a Poisson random measure with constant intensity given by an upper bound on the state-dependent intensity. The jumps of the original state-dependent process are choosen out of the jumps of the new dominating Poisson random measure by a rejection sampling method implemented via a discontinuous jump coefficient in the evolution equation.
In this way an alternative closed form representation to \eqref{random_time_change_eq} of certain PDMPs is obtained as a jump diffusion (without Brownian motion part and) with a discontinuous jump coefficient.  
The main difference to our approach is the following: As it is the case in the use of the thinning procedure, we presently also assume that the rates are bounded. However, this assumption is just a sufficient condition for proving convergence of the numerical approximation, whereas for the thinning procedure it is a necessary condition for the definition of the approximation method. Furthermore, numerical approximations based on thinning decrease in efficiency the wider the rates fluctuate. If the upper bound on the jump rate is much larger than typical values of the state-dependent rate then a large number of jumps are rejected. This decrease in efficiency does not occur for the methods in this study. Here the only influence of the state-dependent rate on the error is its Lipschitz constant, hence the method decreases in efficiency when it is large, i.e., when the rates varies very fast, which is typically referred to as stiffness in numerical analysis.


\bibliographystyle{plain}
\bibliography{bibliography}

\clearpage

\appendix

\section{Martingale problem for Random Time Change Equations}\label{section_app_mart_problem}

In order to establish the equivalence in law of PDMPs and solutions of equations of the form \eqref{random_time_change_eq} we need to establish the following facts. First, necessarily a solution to \eqref{random_time_change_eq} needs to exists. Second, this solution solves the same martingale problem as the PDMP and finally that the solution of the martingale problem is unique. We discuss these points in this appendix, where the last point was established in \cite{Debussche} and is reported in Theorem \ref{Debussche_theorem} below and the first two issues are stated in Theorem \ref{theorem_equivalence}, the proof of which is presented below.\medskip

Recall that throughout the study $(\Omega,\sF,\Pr)$ denotes a complete probability space and further recall that a solution to the martingale problem posed by a generator $\mathcal{A}$ and a characterised family of functions $F:\rr^d\to\rr$ is a stochastic process $X(t)$ satisfying that
\begin{equation}
F(X(t))-F(X(0))-\int_0^t\mathcal{A}F(X(s))\,\dif s
\end{equation}
is a martingale with respect to the natural filtration $(\sF^X_t)_{t\geq 0}$ for any such function $F$, cf.~\cite[p.~173]{EthierKurtz}. The martingale problem is well-posed if for a given initial law a unique solution exists. Here uniqueness refers to the law of the resulting process. We repeat that $C^1_b(\rr^d)$ denotes the set of functions $F:\rr^d\to\rr$ which are bounded and continuously differentiable with bounded derivatives. For the sake of completeness we state the theorem that provides well-posedness of the martingale problem posed by the generator of a PDMP of the class relevant in this study.

\begin{theorem}[\textnormal{\cite[Thm.~2.5]{Debussche}}]\label{Debussche_theorem} Let $f$ and $\lambda_k$ be in $\in C^1_b(\rr^d)$. Then, it holds that the law of a PDMP as described in \textnormal{Section \ref{section_set_up}} is for given initial law the unique solution to the martingale problem posed by the generator
\begin{equation}\label{generator_pdmp_appendix}
\mathcal{A}F(x)=\nabla F[f](x)+\sum_{k=1}^p\lambda_k(x)\,\bigl( F(x+\nu_k)-F(x)\bigr)
\end{equation}
for $F\in C^1_b(\rr^d)$.
\end{theorem}

We now proceed to the discussion of the random time change equation \eqref{random_time_change_eq}, i.e., equations of the form
\begin{equation}\label{appendix_time_change_eq}
X(t)=X(0)+\int_0^tf(X(s))\,\dif s+\sum_{k=1}^pY_k\Bigl(\int_0^t\lambda_k(X(s))\,\dif s\Bigr)\,\nu_k\,,
\end{equation}
where $Y_k$ are independent unit rate Poisson processes defined on $(\Omega,\sF,\Pr)$ and present the proof of Theorem \ref{theorem_equivalence} which we restate for convenience. It provides the well-posedness of the equation \eqref{appendix_time_change_eq} the equivalence of the solution process to a PDMP.

\begin{theorem}\label{theorem_appendix} Let $f$ and $\lambda_k$ be in $\in C^1_b(\rr^d)$, then it holds that equation \eqref{appendix_time_change_eq} possesses for almost all $\omega\in\Omega$ a pathwise unique  c\`adl\`ag solution which defines a stochastic process. Furthermore, this solution process solves the martingale problem posed by the generator \eqref{generator_pdmp_appendix} and initial law given by the law of $X(0)$. 
\end{theorem}

\begin{remark} Prior to presenting the proof of the theorem we state some immediate consequences of this theorem. First, the uniqueness of the martingale problem implies the (strong) Markov property for the solution of equation \eqref{appendix_time_change_eq} with respect to the filtration it generates \cite[Thm.~4.2, Chap.~4]{EthierKurtz}. Furthermore, the solution is non-anticipating, which in the present context means that the increments of the Poisson processes under the random time change are independent. It is clear that these properties still hold when we consider the finer filtration $(\mathcal{G}_t)_{t\geq 0}$ defined in Section \ref{section_set_up} instead of the natural filtration.
\end{remark}

\begin{proof} The proof of the theorem is split into two parts. In the first part (a) we show the existence of a unique solution to equation \eqref{appendix_time_change_eq} iteratively constructing this solution. In the second part (b) we establish that this solution solves the stated martingale problem. This part of the proof follows the outline given in \cite[Sec.~6.4]{EthierKurtz}. Finally, it follows from \cite[Chap.~6, Lemma 2.1]{EthierKurtz} that the collection of pathwise unique solutions defines a stochastic process. \medskip

(a)\quad We prove existence and uniqueness of a solution to the random time change equation \eqref{appendix_time_change_eq}. We set $X_0(t)$ to be the solution of the deterministic part of the equation, i.e., the system $\dot x =f(x)$ with initial condition $x_0=X(0)$ and define the iteration
\begin{equation*}
\nbar X_1(t)=X(0)+\int_0^t f(X_0(s))\dif s+ \sum_{k=1}^pY_k\Bigl(\int_0^t\lambda_k(X_0(s))\,\dif s\Bigr)\nu_k\,,
\end{equation*}
and
\begin{equation*}
X_1(t)\,=\,\left\{\begin{array}{ll}
\nbar X_1(t) & \tn{if } 0\leq t \leq \nbar\tau_1,\\[2ex]
\nbar X_1(\nbar\tau_1)+\displaystyle\int_0^{t-\nbar\tau_1} f(X_1(s))\dif s &\tn{if } \nbar\tau_1<t,
\end{array}\right.
\end{equation*}
where $\nbar\tau_1$ denotes the first jump of the process $\nbar X_1$. Note that the process $\nbar X_1$ solves \eqref{appendix_time_change_eq} up to and including its first jump time and the process $X_1$ solves \eqref{appendix_time_change_eq} up to and excluding the second jump time. We no proceed iteratively setting
\begin{equation*}
\nbar X_n(t)=X(0)+\int_0^t f(X_{n-1}(s))\dif s+ \sum_{k=1}^pY_k\Bigl(\int_0^t\lambda_k(X_{n-1}(s))\,\dif s\Bigr)\nu_k\,,
\end{equation*}
and
\begin{equation*}
X_n(t)\,=\,\left\{\begin{array}{ll}
\nbar X_{n}(t) & \tn{if } 0\leq t\leq \nbar\tau_n,\\[2ex]
\nbar X_n(\nbar\tau_n)+\displaystyle\int_0^{t-\nbar\tau_n} f(X_n(s))\dif s & \tn{if } \nbar\tau_1<t,
\end{array}\right.
\end{equation*}
where $\nbar\tau_n$ denotes the $n$th jump of $\nbar X_n$. As before it holds that $\nbar X_n$ solves \eqref{appendix_time_change_eq} up to and including its $n$th jump time and the process $X_n$ solves \eqref{appendix_time_change_eq} up to and excluding its $(n+1)$th jump time. Note that the boundedness of the $\lambda_k$ implies that $\nbar \tau_n\to\infty$ a.s.~and thus setting $X(t):=\lim_{n\to\infty} X_n(t)$ defines a c\`adl\`ag process which solves \eqref{appendix_time_change_eq} for all times $t\geq 0$. Finally, uniqueness of the solution follows immediately from the uniqueness of the solution to part $\dot x=f(x)$.\medskip

(b) We proceed to show that this process now solves the martingale problem posed by \eqref{generator_pdmp_appendix}. Introducing the (trivial) Markov process $T(t)=t$ yields that the solution process $(X(t))_{t\geq 0}$ satisfies
\begin{eqnarray*}
X(t)&=&X(0)+\sum_{j=1}^d T\Bigr(\int_0^t f_j(X(s))_+\,\dif s\Bigl)\,e_j +\sum_{j=1}^d T\Bigr(\int_0^t f_j(X(s))_-\,\dif s\Bigl)\,(-e_j)\\[1.5ex]
&&\mbox{}+\sum_{k=1}^p Y_k\Bigr(\int_0^t \lambda_k(X(s))\,\dif s\Bigl)\,\nu_k\,,
\end{eqnarray*}
where the subscripts $+/-$ denote the positive and negative part of a function, respectively, and $e_j$ stands for the $j$th unit vector in $\rr^d$. We define for $z\in\rr^{p+2d}$ the map
\begin{equation}
\Gamma(z)=X(0)+\sum_{l=1}^p z_l\,\nu_l+\sum_{l=p+1}^{p+d} z_l\,(e_j)+\sum_{l=p+d+1}^{p+d} z_l\,(-e_j)\,.
\end{equation}
and using the definition of a process $Z_l$, $l=1,\ldots,p+2d$, given by 
\begin{equation*}
Z_l(t)=\left\{\begin{array}{cl}
Y_l\displaystyle\Bigl(\int_0^t\lambda_l(X(s))\,\dif s\Bigr)& \textnormal{if } l=1,\ldots,p,\\[2ex]
T\displaystyle\Bigl(\int_0^t f_{l-p}(X(s))_+\,\dif s\Bigr)& \textnormal{if } l=p+1,\ldots,p+d,\\[2ex]
T\displaystyle\Bigl(\int_0^t f_{l-p-d}(X(s))_-\,\dif s\Bigr)& \textnormal{if } l=p+d+1,\ldots,p+2d\,,
\end{array}\right.
\end{equation*}
we obtain that $X(s)=\Gamma(Z(s))$ and therefore
\begin{equation}
Z_l(t)=\left\{\begin{array}{cl}
Y_l\displaystyle\Bigl(\int_0^t\lambda_l(\Gamma(Z(s)))\,\dif s\Bigr)& \textnormal{if } l=1,\ldots,p,\\[2ex]
T\displaystyle\Bigl(\int_0^t f_{l-p}(\Gamma(Z(s)))_+\,\dif s\Bigr)& \textnormal{if } l=p+1,\ldots,p+d,\\[2ex]
T\displaystyle\Bigl(\int_0^t f_{l-p-d}(\Gamma(Z(s)))_-\,\dif s\Bigr)& \textnormal{if } l=p+d+1,\ldots,p+2d\,.
\end{array}\right.
\end{equation}
We now infer from the results in \cite[Sec.~6.2]{EthierKurtz} that the process $(Z(t))_{t\geq 0}$ is the unique solution to the martingale problem posed by functions of the form $F(z)=\prod_{l=1}^{p+2d} F_i(z_i)$, $F_i\in C^1_b(\rr)$, and the generator
\begin{eqnarray}\label{martingale_problem_random_time}
\lefteqn{\mathcal{A}_ZF(z)\ = }\nonumber\\
&=&\sum_{l=1}^p \lambda_l(\Gamma(z))\Bigl( F_l(z_l+1)-F_l(z_l)\Bigl)\prod_{j\neq l}F_j(z_j)+\sum_{l=p+1}^{p+d} \Bigl(f_{l-p}(\Gamma(z))_+\nabla F_{l-p}(z_l)\Bigr)\prod_{k\neq l-p}F_j(z_j)\nonumber\\[2ex]
&&\mbox{}- \sum_{l=p+d+1}^{p+2d} \Bigl(f_{l-p-d}(\Gamma(z))_-\nabla F_{l-p-d}(z_l)\Bigr)\prod_{k\neq l-p-d}F_j(z_j)\,.
\end{eqnarray}
Moreover, this process also solves the martingale problem posed by the linear span of such functions and the bp-closure thereof, cf.~\cite[p.~174]{EthierKurtz}, which contains any bounded, continuously differentiable function on $\rr^d$ with bounded first derivatives. 
Particularly, the bp-closure contains $F\circ\Gamma$ for any $F\in C^1_b(\rr^d)$. Inserting such a function into the generator \eqref{martingale_problem_random_time} we find for its continuous part
\begin{eqnarray*}
&&\hspace{-20pt}\sum_{l=p+1}^{p+d} f_{l-p}(X(s))_+\nabla_{z_l} F\Bigl(X(0)+\sum_{j=1}^p Z_j(s)\,\nu_j+\sum_{j=p+1}^{p+d} Z_j(s)\,(e_{j-p})+\sum_{j=p+d+1}^{p+2d} Z_j(s)\,(-e_{j-p-d})\Bigr)\nonumber\\[2ex]
&&\hspace{-20pt}\mbox{}- \sum_{l=p+d+1}^{p+2d} f_{l-p-d}(X(s))_-\nabla_{z_{l-p-d}} F\Bigl(X(0)+\sum_{j=1}^p Z_j(s)\,\nu_j+\sum_{j=p+1}^{p+d} Z_j(s)\,(e_{j-p})\\[2ex]
&&\phantom{xxxxxxxxxxxxxxxxxxxxxxxxxxxxxxxxxxxxxxxxxxxxxx}\mbox{}+\sum_{j=p+d+1}^{p+2d} Z_j(s)\,(-e_{j-p-d})\Bigr)\Bigr)\\[2ex]
&&\hspace{-20pt}=\,\sum_{l=1}^d f_j(X(s))_+\nabla F[e_j](X(s)-\sum_{l=1}^d f_j(X(s))_-\nabla F[e_j](X(s))\\[2ex]
&&\hspace{-20pt}=\,\sum_{j=1}^d \nabla F[f_j(X^0(s))\,e_j](X(s))\\[2ex]
&&\hspace{-20pt}=\, \nabla F[f](X(s))
\end{eqnarray*}
and for its jump part
\begin{equation*}
\sum_{l=1}^p \lambda_l(\Gamma(Z(s)))\Bigl( F(\Gamma(Z(s)+e_l))-F(\Gamma(Z(s)))\Bigl)\,=\,
\sum_{l=1}^p \lambda_l(X(s))\Bigl(F(X(s)+\nu_l)-F(X(s))\Bigr)\,.
\end{equation*}
Therefore, we infer the solution $(X(t))_{t\geq 0}$ to \eqref{appendix_time_change_eq} solves the martingale problem posed by the generator
\begin{equation*}
\mathcal{A}F(x)=\nabla F[f](x)+\sum_{k=1}^p\lambda_k(x)\,\bigl( F(x+\nu_k)-F(x)\bigr)
\end{equation*}
with $F\in C^1(\rr^d)$. However, this is just the martingale problem posed by \eqref{generator_pdmp_appendix} and the proof of the theorem is completed.
\end{proof}

\begin{remark} We note that we expect the above theorems to remain valid under less restrictive assumptions on the characteristics $f,\,\lambda_k$ of the PDMP, when instead of the martingale the local martingale problem is considered. E.g., in the case of unbounded coefficients a time change analogously to \cite[Sec.~6.4]{EthierKurtz} first reduces the problem to an equation with bounded coefficients. Then inverting the time change establishes the result for the original local martingale problem. Similarly, it is further explained in \cite{Debussche} that also for the well-posedness of the martingale problem the assumptions can be already weakened. For the existence of a solution to \eqref{appendix_time_change_eq} it is in general sufficient that $\dot x=f(x)$ possesses a unique global solution for any initial condition and that $\int_0^t\sum_{k=1}^p\lambda_k(X_s)<\infty$ a.s.~for all $t\geq 0$. However, as in this study boundedness of the coefficients and their derivatives is always assumed, we chose the simpler presentation in order to focus on the essential and avoid overcomplicating matters.
\end{remark}

\end{document}